\documentclass{amsart}
\usepackage{todonotes}
\usepackage{xypic}
\usepackage{graphicx, amsmath, amssymb, amsthm}
\usepackage{hyperref} 

\newcommand{\NSFThree}{NSF Grant DMS-1509652}
\newcommand{\Sloan}{the Sloan Foundation}

\newcommand{\Z}{{\mathbb  Z}}
\newcommand{\N}{{\mathbb N}}

\newcommand{\Ind}{\big\uparrow} 
\newcommand{\ind}{\uparrow}

\DeclareMathOperator{\Map}{Map}
\DeclareMathOperator{\Sym}{Sym}

\DeclareMathOperator{\Stab}{Stab}

\newcommand{\res}{res}

\newcommand{\m}[1]{{\protect\underline{#1}}}

\newcommand{\mM}{\m{M}}

\newcommand{\mR}{\m{R}}

\newcommand{\mA}{\m{A}}

\newcommand{\mSet}{\m{\Set}}

\newcommand{\cc}[1]{\mathcal #1}
\newcommand{\cC}{\cc{C}}
\newcommand{\ccD}{\cc{D}}
\newcommand{\cF}{\cc{F}}
\newcommand{\cO}{\cc{O}}

\newcommand{\cP}{\cc{P}}

\newcommand{\aC}{\cC}
\newcommand{\aD}{\ccD}
\newcommand{\ob}{\textrm{ob}}
\newcommand{\id}{\textrm{id}}

\newcommand{\CoInd}{\textnormal{CoInd}}

\newcommand{\Set}{\mathcal Set}

\newcommand{\Ninfty}{N_\infty}

\newcommand{\cOrb}{\mathcal Orb}

\newcommand{\Tamb}{\mathcal Tamb}
\newcommand{\OTamb}{\cO\mhyphen\Tamb}
\newcommand{\Mackey}{\mathcal Mackey}
\newcommand{\Cat}{\mathcal Cat}
\newcommand{\SymM}{\mathcal Sym}

\newcommand{\Green}{\mathcal Green}

\newcommand{\leftadjoint}{\dashv}

\newcommand{\mC}{\m{\cC}}

\mathchardef\mhyphen=45

\newtheorem{theorem}{Theorem}[section]
\newtheorem{lemma}[theorem]{Lemma}
\newtheorem{corollary}[theorem]{Corollary}

\newtheorem{definition}[theorem]{Definition}
\newtheorem{proposition}[theorem]{Proposition}

\newtheorem{remark}[theorem]{Remark}
\newtheorem{example}[theorem]{Example}

\newtheorem{warning}[theorem]{Warning}

\newcommand{\defemph}[1]{\textbf{#1}} 
\newcommand{\Ho}{\textrm{Ho}}

\begin{document}

\title{Incomplete Tambara Functors}

\author[A.~J.~Blumberg]{Andrew~J. Blumberg}
\address{University of Texas \\ Austin, TX 78712}
\email{blumberg@math.utexas.edu}
\thanks{A.~J.~Blumberg was supported in part by NSF grant DMS-1151577}

\author[M.~A.~Hill]{Michael~A. Hill}
\address{University of California Los Angeles\\ Los Angeles, CA 90095}
\email{mikehill@math.ucla.edu}
\thanks{M.~A.~Hill was supported in part by {\NSFThree} and {\Sloan}}

\begin{abstract}
For a ``genuine'' equivariant commutative ring spectrum $R$,
$\pi_0(R)$ admits a rich algebraic structure known as a Tambara
functor.  This algebraic structure mirrors the structure on $R$
arising from the existence of multiplicative norm maps.  Motivated by
the surprising fact that Bousfield localization can destroy some of
the norm maps, in previous work we studied equivariant commutative
ring structures parametrized by $N_\infty$ operads.  In a precise
sense, these interpolate between ``naive'' and ``genuine'' equivariant
ring structures. 

In this paper, we describe the algebraic analogue of $N_\infty$ ring
structures.  We introduce and study categories of incomplete Tambara
functors, described in terms of certain categories of bispans.
Incomplete Tambara functors arise as $\pi_0$ of $N_\infty$ algebras,
and interpolate between Green functors and Tambara functors.  We
classify all incomplete Tambara functors in terms of a basic
structural result about polynomial functors.  This classification
gives a conceptual justification for our prior description of
$N_\infty$ operads and also allows us to easily describe the
properties of the category of incomplete Tambara functors.  
\end{abstract}

\maketitle

\section{Introduction}

Much of the richness and subtlety of equivariant stable homotopy
theory arises from the complexity of the notion of a commutative ring
spectrum (i.e., multiplicative cohomology theory) in this context.
Although one can define an equivariant commutative ring spectrum as an
equivariant spectrum with a homotopy-coherent multiplication
parametrized by an $E_\infty$ operad (regarded as a $G$-trivial
equivariant operad), much more power comes from considering
multiplications parametrized by ``genuine'' equivariant $E_\infty$
operads.  Such commutative ring spectra have a coherent collection of
multiplicative norm maps, studied extensively first by Greenlees and
May~\cite{GreenMay} and utilized to great effect in the work of the
second author, Hopkins, and Ravenel~\cite{HHR} resolving the Kervaire
invariant one problem.

One of the most surprising observations emerging from the recent
renewed interest in equivariant stable homotopy is the discovery that
Bousfield localization does not necessarily preserve the existence of
these multiplicative norms~\cite{HHInversion, HHLocalization}.
Succinctly, Bousfield localization does not necessarily take genuine
equivariant commutative ring spectra to genuine equivariant
commutative ring spectra.  For formal reasons, the localization of a
genuine equivariant commutative ring spectra is equipped with a
homotopy-coherent multiplication, but the question of which norm maps
survive is considerably more complicated.

In order to understand exactly what kinds of structures are preserved,
in previous work we introduced the more general notion of an $\Ninfty$
operad, an operad in $G$-spectra which interpolates between the naive
$E_{\infty}$ operads which parameterize a coherently homotopy
commutative multiplication and the genuine $G$-$E_{\infty}$ operads
whose algebras are genuine equivariant commutative ring
spectra~\cite{BHNinfty}.  Algebras in spectra over $\Ninfty$ operads are
equivariant commutative ring spectra that admit families of
multiplicative norms.  One of the most surprising results in our study
of $\Ninfty$ operads and their algebras was that homotopically, the
entire story is essentially discrete. The homotopy type of an
$\Ninfty$ operad is completely determined by an ``indexing system'', a
coherent collection of finite $H$-sets for each subgroup $H$ of $G$.
In addition to their role classifying $\Ninfty$ operads, indexing
systems also parameterize exactly which norms arise in algebras over
an $\Ninfty$ operad.  For example, the trivial $\Ninfty$ operad gives
only a coherently commutative multiplication, while a $G$-$E_\infty$
operad gives compatible norm maps for all pairs of subgroups $H\subset K$.  

\begin{definition}
A \defemph{symmetric monoidal coefficient system} is a contravariant functor 
\[
\mC\colon\cOrb_{G}^{op}\to\SymM
\]
from the opposite of the orbit category of $G$ to the category of symmetric monoidal categories and strong symmetric monoidal functors.
\end{definition}

The prototype of a symmetric monoidal is $\mSet$, the functor which assigns to $G/H$ the category of finite $H$-sets, viewed as a symmetric monoidal category under disjoint union. The functoriality here is most easily seen by replacing $\Set^H$ with the equivalent category of finite $G$-sets over $G/H$, and we will implicitly work in this formulation.
An indexing system is a sub coefficient systems of this coefficient system which has properties analogous to closure under composition.

\begin{definition}\label{defn:IndexingSystem}
An {\defemph{indexing system}} is a full symmetric monoidal sub-coefficient system $\mC$ of $\mSet$ that contains all trivial sets and is closed under
\begin{enumerate}
\item finite limits and
\item ``self-induction'': if $H/K\in \mC(H)$ and $T\in\mC(K)$, then $H\times_{K}T\in\mC(H)$.
\end{enumerate}
To reduce clutter, we  write $\mC(H)$ for $\mC(G/H)$.
\end{definition}

The set $\mathcal I$ of indexing systems forms a poset under inclusion, and one of the basic results in \cite{BHNinfty} is that there is a fully-faithful functor
\[
\mC\colon \Ho\left(\mathcal N_{\infty}\mhyphen{\mathcal O}perad\right)\to \mathcal I.
\]
We conjecture there that this functor is in fact an equivalence of categories.

The purpose of this paper is to study the analogous story in algebra,
which provides a conceptual explanation of the homotopically discrete
behavior of $N_\infty$ operads.  Via $\pi_0$, the structure of the
equivariant stable category is mirrored in the abelian category of
Mackey functors.  Mackey functors which have a commutative
multiplication, typically referred to as commutative Green functors,
mirror the structure of a homotopy-coherent multiplication on an
equivariant spectrum.  Mackey functors that admit a commutative
multiplication and in addition multiplicative norm maps, known as
Tambara functors, mirror the structure of a genuine equivariant
commutative ring spectrum.  Although the theory of these sorts of
algebraic equivariant ring objects is well-developed, there has not
been any study of the algebraic analogue of the algebras over the
intermediate $\Ninfty$ operads, namely commutative Green functors
which have some, but not necessarily all, multiplicative norm
maps. This paper introduces these ``incomplete Tambara functors'' and
explores their basic properties.

Tambara originally defined his $TNR$-functors as product-preserving functors from a category of ``bispans'', now called ``polynomials'', of finite $G$-sets into the category sets \cite{Tambara}. Here, a ``bispan'', is an isomorphism class of diagrams of the form
\begin{equation}\label{eqn:BiSpan}
X\xleftarrow{h} A\xrightarrow{g} B\xrightarrow{f} Y,
\end{equation}
where isomorphisms are isomorphisms of diagrams which are the identity on $X$ and on $Y$. The set of all such isomorphisms forms the morphisms from $X$ to $Y$ in the category of bispans, and in this category, disjoint union of finite $G$-sets forms the product in this category. Category theorists have generalized this approach, describing the category of polynomials in a wide variety of contexts such as locally Cartesian closed categories or categories with pullbacks. 
Any arrow in the category of polynomials in a category $\cC$ can be
decomposed as $T_f\circ N_g\circ R_h$, where the definitions of these
maps are reviewed below. The choices of letters reflect the underlying
structure: $R$ gives the restriction in a Mackey functors, $T$ the
transfer, and $N$ the norm, a generalization of the Evans transfer in
group cohomology. Our incomplete Tambara functors arise by restricting the map $g$
in Equation~\ref{eqn:BiSpan} to live in a particular subcategory.  The
fact that such a restriction is well-defined comes from the following
theorem, which holds for polynomials in any of the contexts normally
studied.

\begin{theorem}
Let $\cC$ be a locally Cartesian closed category or more generally a category with pullbacks, and let $\ccD$ be a wide, pullback stable subcategory of $\cC$. Then the subgraph of the category of polynomials in $\cC$ with all objects and with only morphisms of the form
\[
X\xleftarrow{h} A\xrightarrow{g} B\xrightarrow{f} Y,
\]
where $g\in\ccD$, is a subcategory.
\end{theorem}

We call this subcategory the ``polynomials in $\cC$ with exponents in $\ccD$''. This theorem shows that we can find interesting generalizations of Tambara's construction by considering the wide, pullback stable subcategories of the category of finite $G$-sets. We have a complete classification of the wide, pullback stable subcategories which are of most interest to us, and this is the major result of Section~\ref{sec:PBStableSubcats}.

\begin{theorem}
There is an isomorphism between the poset of indexing systems and the poset of wide, pullback stable, finite coproduct complete subcategories of $\Set^G$.
\end{theorem} 

Loosely speaking, an indexing system describes all of the norm maps
which arise in the study of algebras over an $\Ninfty$ operad, and
this theorem shows that from a categorical point of view, these are
all that we should have expected. 

Functors out of the category of polynomials in $\Set^G$ with exponents
in various wide, pullback stable subcategories gives our notion of
incomplete Tambara functors. 

\begin{definition}[See Definition~\ref{defn:incompleteTamb}]
For an indexing system $\cO$, the category of $\cO$-Tambara functors
is the category of functors from polynomials with exponents in $\cO$
to abelian groups.
\end{definition}

In Sections~\ref{sec:OTambara} and~\ref{sec:OTambprop}, we describe
the basic constructions and explores some of the properties of the
category of incomplete Tambara functors. In particular, we describe
also the result of localization in the category of incomplete Tambara
functors, building on work of Nakaoka \cite{NakaokaLocalization}.

Next, in Section~\ref{sec:ChangeofGroup} we study ``change'' functors
and in particular focus on the analogue of the norm-restriction
adjunction in this context.

Finally, we explain the connection between $\cO$-ring spectra and
$\cO$-Tambara functors, for $\cO$ an $N_\infty$ operand.  We have the
following theorem, proved as Theorem~\ref{thm:pi0} below.

\begin{theorem}
Let $\cO$ be an $N_\infty$ operad and $R$ an $\cO$-algebra in
orthogonal $G$-spectra.  Then $\pi_0(R)$ is an $\cO$-Tambara functor.
\end{theorem}

In fact, the functor $\pi_0$ translates the $G$-symmetric monoidal
structure on the equivariant stable category associated to an
$N_\infty$ operad $\cO$ (as in~\cite{BHmodules}) to the $G$-symmmetric
monoidal structure on Mackey functors specified by $\cO$.  However, to
be precise about this, we need to study the homological algebra of
$\cO$-Tambara functors.  More generally, because they are central to
the behavior of localization on commutative rings in equivariant
stable homotopy theory, we expect that the theory of the homological
algebra of $\cO$-Tambara functors will be an important aspect of
developing equivariant derived algebraic geometry.  We intend to carry
out this work in a subsequent paper.

\begin{remark}
We work in this paper in an additively complete setting, meaning that all of our flavors of Tambara functors will have an underlying Mackey functor. This is motivated by our goals in equivariant spectra, where our objects of study are multiplicative structures put on genuine $G$-spectra. Allowing additive incompleteness as well adds very interesting consequences, and we will return to this in a future paper.
\end{remark}

\subsection*{Acknowledgments.}  The authors would like to thank John Greenlees, Mike 
Mandell, and Peter May for helpful conversations.  This project was
made possible by the hospitality of the Hausdorff Research Institute
for Mathematics at the University of Bonn.

\subsection{Notation}

We will use the symbol $\cO$ abusively to refer either to an
$\Ninfty$ operad or to an indexing system; when $\cO$ refers to an
$\Ninfty$-operad, we will use the same symbol to describe the
associated indexing system.  If $\cO$ is an indexing system, then we
say that an $H$-set $T$ is ``admissible for $\cO$'' if $T\in \cO(H)$.

If $S$ is a $G$-set and $s\in S$, let $G_{s}=\Stab(s)$ denote the
stabilizer subgroup. 

\section{Polynomials with restricted exponents}\label{sec:Polynomials}

A polynomial (or bispan) in a category $\aC$ is an isomorphism class
of composites $X \leftarrow S \to T \to Y$.  The collection of
polynomials forms a category with composition given by pullback.
Tambara functors can be described as certain Functors out of this
category.  Given a subcategory $\aD \subset \aC$, a natural question
when considering incomplete Tambara functors is to consider
polynomials in $\aC$ with the ``middle'' map $S \to T$ required to be
in $\aD$.  In this section, we discuss the basic theory of polynomials
with restricted exponents; we also establish some technical results
about polynomials with restricted exponents that we need in the
remainder of the paper.  The main result of this section
(Theorem~\ref{thm:PolynomialSubcategories}) provides natural criteria
on $\aD$ that describe when the resulting collection of polynomials
itself forms a category.

\subsection{Review of Polynomials}
Much of the background material here is taken from work of
Gambino-Kock, although everything works in Weber's context as well \cite{GamKock, Weber}.  Following their conventions, we
work in a locally Cartesian closed category $\cC$. In particular, this
means that for any morphism $f\colon X\to Y$, the pullback functor
$f^{\ast}\colon \cC/Y\to\cC/X$ has both adjoints: 
\[
\Sigma_{f}\dashv f^{\ast}\dashv \Pi_{f}.
\]
The left adjoint is called the ``dependent sum'' and the right the
``dependent product''. In the category of finite $G$-sets, the
dependent sum is simply ``disjoint union of the fibers over $y$'',
while the dependent product is the ``product of the fibers over
$y$''. 

In any locally Cartesian closed category $\cC$, we can define the category of polynomials.

\begin{definition}
If $\cC$ is a locally Cartesian closed category, let $\cP^{\cC}$ be the category with objects the objects of $\cC$ and with morphisms isomorphism classes of ``bispans''
\[
X\leftarrow S\to T\to Y.
\]
Here isomorphisms of bispans are specified by isomorphisms $S \to S'$ and $T \to T'$ such that following diagram commutes
\[
\xymatrix@R=.5\baselineskip{
{} & {S}\ar[r]\ar[dd]_{\cong}\ar[dl] & {T}\ar[dd]^{\cong}\ar[dr] & {} \\
{X} & {} & {} & {Y.} \\
{} & {S'}\ar[ul]\ar[r] & {T'} \ar[ur] & {}
}
\]
\end{definition}

\begin{remark}
In fact, we have a bicategory of polynomials in $\cC$, where the diagram expressing an isomorphism defines a $2$-cell provided the central square is a pullback.  In all of what follows, the statements remain true if we work in this $\Cat$ enriched setting.  
\end{remark}

Composition of bispans in most easily expressed by choosing a convenient set of generating morphisms.

\begin{definition}
If $f\colon S\to T$ is a map in $\cC$, then let
\begin{enumerate}
\item $R_f=[T\xleftarrow{f} S\xrightarrow{1} S\xrightarrow{1} S]\in \cP^{\cC}(T,S),$
\item $N_f=[S\xleftarrow{1} S\xrightarrow{f} T\xrightarrow{1} T]\in \cP^{\cC}(S,T),$ and
\item $T_f=[S\xleftarrow{1} S\xrightarrow{1} S\xrightarrow{f} T]\in \cP^{\cC}(S,T)$.
\end{enumerate}
We will refer to maps of this form as {\em basic} polynomials.
\end{definition}

These maps generate the category $\cP^{\cC}$.  First, we have an identification
\begin{equation}\label{eqn:aTNR}
[X\xleftarrow{f} S\xrightarrow{g} T\xrightarrow{h} Y] =T_h\circ N_g\circ R_f.
\end{equation}

We will say that the order of maps $T$, $N$, $R$ given by
Equation~\ref{eqn:aTNR} is the ``canonical ordering'', and we need only
show that any other composite of basic maps can be brought into this
form.  We do this by establishing commutation relations.  The argument
goes back to Tambara in the context of finite $G$-sets
\cite{Tambara}. In the locally Cartesian closed category context,
these were shown by Gambino-Kock and in the context of categories with
pullbacks, by Weber. We summarize the commutation relations in a
series of propositions. 

\begin{proposition}[{\cite{Tambara,GamKock,Weber}}]\label{prop:TNRcomp}
We have
\begin{align*}
N_{g}\circ N_{g'}&=N_{g\circ g'}\\
T_{h}\circ T_{h'}&= N_{h\circ h'}\\
R_{f}\circ R_{f'}&=R_{f'\circ f}.
\end{align*}
\end{proposition}

\begin{proposition}[{\cite{Tambara,GamKock,Weber}}]\label{prop:RN}
If 
\[
\xymatrix{
{X'}\ar[r]^{g'}\ar[d]_{f'} &  {X}\ar[d]^{f} \\
{Y'}\ar[r]_{g} & {Y}
}
\]
is a pullback diagram, then we have
\begin{align*}
R_{f}\circ N_{g}&=N_{g'}\circ R_{f'}\\
R_{f}\circ T_{g} &=T_{g'}\circ R_{f'}.
\end{align*}
\end{proposition}

\begin{proposition}[{\cite{Tambara,GamKock,Weber}}]\label{prop:NT}
If
\[
\xymatrix{
{X}\ar[d]_{g} & {A}\ar[l]_{h} & {X\times_{Y}\prod_{g} A}\ar[d]^{g'}\ar[l]_-{f'} \\
{Y} & & {\prod_{g} A}\ar[ll]^{h'}
}
\]
is an exponential diagram, then we have
\[
N_{g}\circ T_{h}=R_{f'}\circ N_{g'}\circ T_{h'}.
\]
\end{proposition}

For our purposes, a key fact is that the outer rectangle of an exponential diagram is actually a pullback diagram.  

\subsection{Polynomials with restricted exponents}
We can now describe several natural subcategories of the category of polynomials.  Recall that a subcategory of $\cC$ is {\em wide} if it contains all of the objects, and {\em essentially wide} if every object of $\cC$ is isomorphic to an object in the subcategory.

\begin{definition}\label{adef:Exponents}
If $\ccD\subset \cC$ is a wide subcategory, then let $\cP_{\ccD}^{\cC}$ by the wide subgraph of $\cP^{\cC}$ with morphisms the isomorphism classes of bispans
\[
X\leftarrow S\xrightarrow{f} T\to Y,
\]
with $f\in \ccD$. We call this the {\defemph{polynomials in $\cC$ with exponents in $\ccD$}}.
\end{definition}

The somewhat surprising result is that $\cP_{\ccD}^{\cC}$ is a subcategory of $\cP^{\cC}$ under the hypothesis of pullback stability. We begin by recalling the notion of a pullback stable subcategory.

\begin{definition}
If $\cC$ is a category that admits pullbacks, then we say that a subcategory $\ccD\subset\cC$ is \defemph{pullback stable} if whenever
\[
\xymatrix{
{A}\ar[d]_f \ar[r] & {B}\ar[d]^g \\
{C}\ar[r] & {D}
}
\]
is a pullback diagram and $g\in\ccD$, the map $f$ is also in $\ccD$.
\end{definition}

We have two elementary results which we use quite often.  

\begin{proposition}\label{prop:basic}
Let $\ccD$ be a pullback stable subcategory of a category $\cC$ that admits pullbacks.
\begin{enumerate}
\item If $A\in\ob(\ccD)$ and $f\colon B\to A$ is an isomorphism in $\cC$, then $f$ is in $\ccD$.
\item If $\cC$ contains a terminal object $\ast$, then $\ccD$ is a wide subcategory.
\end{enumerate}
\end{proposition}

\begin{proof}
For the first part, observe that 
\[
\xymatrix{
B \ar[r]^-f \ar[d]_-f & A \ar[d]^{\id} \\
A \ar[r]_-{\id} & A\\
}
\]
is a pullback diagram.  For the second, observe that
\[
\xymatrix{
B \ar[r] \ar[d]_-{\id} & \ast \ar[d]^-{\id} \\
B \ar[r] & \ast \\
}
\]
is a pullback diagram.
\end{proof}

With this, we can prove a surprising result that $\cP^{\cC}_{\ccD}$ is actually a subcategory when $\ccD$ is pullback stable and wide.

\begin{theorem}\label{thm:PolynomialSubcategories} 
If $\ccD$ is a wide, pullback stable subcategory, then the subgraph $\cP^{\cC}_{\ccD}$ is a subcategory of $\cP^{\cC}$.
\end{theorem}

\begin{proof}
We show this using the generating morphisms $R_{f}$, $N_{g}$, and $T_{h}$. Since by assumption $\ccD$ is wide, for any $f$ and $h$, $R_f$ and $T_h$ are also in $\cP^{\cC}_{\ccD}$, while $N_g$ is if and only if $g\in\ccD$.  Therefore, it suffices to show that the composite of any such morphisms is again of the form
\[
T_{h}\circ N_{g}\circ R_{f}
\]
where $g\in\ccD$.

Since the maps $f$ and $h$ that specify $T_{f}$ and $R_{h}$ are arbitrary morphisms in $\cC$, any composite involving only basic maps of this form will again be an element in $\cP^{\cC}_{\ccD}$.  Thus, we need to show that composites with $N_{g}$ for $g\in \ccD$ are again in $\cP^{\cC}_{\ccD}$. Since $\ccD$ is a subcategory, we have $N_{g}\circ N_{g'}$ is again in $\cP^{\cC}_{\ccD}$ provided $g,g'\in\ccD$.

Proposition~\ref{prop:NT} shows that $N_{g}\circ T_{f}$ can be written as $T_{f'}\circ N_{g'}\circ R_{h'}$, where $g'$ is the pullback of $g$ in the relevant exponential diagram. This shows that $N_{g}\circ T_{f}$ is again in $\cP^{\cC}_{\ccD}$.

Finally, Proposition~\ref{prop:RN} shows that $R_{f}\circ N_{g}=N_{g'}\circ R_{f'}$, where $g'$ is the pullback of $g$ along $f$. This shows that $R_{f}\circ N_{g}$ is again in $\cP^{\cC}_{\ccD}$, and thus it is a subcategory.
\end{proof}

\begin{corollary}\label{cor:PolynomialInclusion}
If $\ccD_{1}\subset\ccD_{2}\subset \cC$ are wide, pullback stable subcategories, then we have an inclusion of subcategories
\[
\cP^{\cC}_{\ccD_{1}}\subset\cP^{\cC}_{\ccD_{2}}.
\]
\end{corollary}

\begin{proposition}\label{prop:Products}
If $\ccD$ is a wide, pullback stable, symmetric monoidal subcategory
of $\Set^{G}$, then $\cP^{G}_{\ccD}$ has finite products and
the products are created in $\cP^{G}$.
\end{proposition}

\begin{proof}
The product in $\cP^{G}$ is induced by the disjoint union of $G$-sets: if $S$ and $T$ are $G$-sets and $i_{S}\colon S\to S\amalg T$ and $i_{T}\colon T\to S\amalg T$ are the inclusions, then 
\[
S\xleftarrow{R_{i_{S}}} S\amalg T\xrightarrow{R_{i_{T}}} T
\]
is a product diagram in $\cP^{G}$. We must show that if $F=T_{h}\circ N_{g}\circ R_{f}$ is any morphism in $\cP^{G}$ then $R_{i_{S}}\circ F, R_{i_{T}}\circ F\in\cP^{G}_{\ccD}$ if and only if $F\in\cP^{G}_{\ccD}$. 

By assumption, $F$ is a polynomial of the form
\[
A\xleftarrow{f} B\xrightarrow{g} C\xrightarrow{h} S\amalg T.
\]
Since we are considering equivariant maps and since we are mapping into a disjoint union, $C$ decomposes as $C_{0}\amalg C_{1}$, where $h(C_{0})\subset S$ and $h(C_{1})\subset T$. Our map $h$ is then the disjoint union of maps $h_{i}=h|_{C_{i}}$. Similarly, $B$ and $g$ decompose as $B=B_{0}\amalg B_{1}$, where $g(B_{i})\subset C_{i}$ and if $g_{i}=g|_{B_{i}}$, then $g=g_{0}\amalg g_{1}$.

We now can directly compute $R_{i_{S}}\circ F$:
\[
R_{i_{S}}\circ T_{h}\circ N_{g}\circ R_{f}  = T_{h_{0}}\circ R_{i_{C_{0}}}\circ N_{g}\circ R_{f} = T_{h_{0}}\circ N_{g_{0}}\circ R_{i_{B_{0}}}\circ R_{f}.
\]
Thus $R_{i_{S}}\circ F$ is in $\cP^{G}_{\ccD}$ if and only if $g_{0}\in\ccD$ and similarly for $R_{i_{T}}\circ F$. Since $\ccD$ is a pullback stable symmetric monoidal subcategory, $g=g_{0}\amalg g_{1}$ is in $\ccD$ if and only if $g_{0}, g_{1}\in\ccD$.
\end{proof}

\begin{remark}
Proposition~\ref{prop:Products} holds much more generally: if $\cC$ is
a disjunctive category, then the same argument given goes through.
\end{remark}

\subsection{Adjunctions between categories of polynomials}
We can determine sufficient conditions for when adjunctions in the ambient categories give rise to adjunctions in the polynomials with exponents in a suitable subcategory. Our motivation is generalizing the classical result that adjoint pair 
\[
\Ind_{H}^{G}\colon \Set^{H}\leftrightarrows \Set^{G}\colon i_{H}^{\ast}
\]
induces by pre-composition an adjoint pair on the categories of Mackey functors:
\[
i_{H}^{\ast}=(\Ind_{H}^{G})^{\ast}\colon \Mackey_{G}\leftrightarrows \Mackey_{H}\colon \CoInd_{H}^{G}=(i_{H}^{\ast})^{\ast}.
\]
Strickland shows that this holds in the categories of Tambara functors, a result we will generalize in Theorem~\ref{thm:RightAdjoint} below. For now, we continue to work in the more abstract context.

\begin{definition}
We say that a subcategory $\ccD\subset \cC$ is \defemph{essentially a sieve} if for all $f'\colon a'\to b$ in $\cC$ and for all $g\colon b\to c$ in $\ccD$, there is an isomorphism $a'\to a$ and a map $f\colon a\to b$ in $\ccD$ making a commutative diagram
\[
\xymatrix{
{a'}\ar[d]_{\cong}\ar[r]^{f'} & {b}\ar[r]^{g} & {c.} \\
{a}\ar[ur]_{f}
}
\]
\end{definition}

\begin{remark}
This is the strong form of the ``not evil'' version of a sieve. Instead of asking that the maps form a kind of ideal under composition, we ask instead that that it be closed under precomposition with an isomorphic map. In particular, it is not a wide subcategory of the slice category in $\cC$ over its objects, but it is essentially wide.
\end{remark}

The prototypical example comes from the category of finite $G$-sets.

\begin{proposition}\label{prop:IndEssentiallyaSieve}
The image of the induction functor $\ind_{H}^{G}$ is essentially a sieve in $\Set^{G}$.
\end{proposition}
\begin{proof}
We must show that if $f'\colon T'\to \Ind_{H}^{G}S$ is any $G$-map, then there is an $H$-set $T$, an $H$-map $f\colon T\to S$, and an isomorphism $T'\to \Ind_{H}^{G} T$ such that the diagram
\[
\xymatrix{
{T'}\ar[r]^{f'}\ar[d]_{\cong} & {\Ind_{H}^{G}S} \\
{\Ind_{H}^{G}T}\ar[ur]_{\ind f}
}
\]
commutes. Let $T$ be the pullback in $H$-sets
\[
\xymatrix{
{T}\ar[r]\ar[d] & {S}\ar[d] \\
{i_{H}^{\ast}T'}\ar[r]_{f'} & {i_{H}^{\ast}\Ind_{H}^{G}S,}
}
\]
where the map from $S$ is the unit of the adjunction. Let $f$ be the map $T\to S$ given by the pullback, and then by construction, the desired diagram commutes. By checking on orbits, we see that the natural map $\ind_{H}^{G}T\to T'$ is also an equivariant isomorphism.
\end{proof}

With this definition, we can describe sufficient conditions for a pair of adjoint functors on $\cC$ to descend to a pair of adjoint functors on $\cP^{\cC}_{\ccD}$.

\begin{theorem}\label{thm:AdjointPair}
Let $F\colon \cC\rightleftarrows \cC'\colon G$ be an adjoint pair with the following properties:
\begin{enumerate}
\item The image of $F$ is essentially a sieve,
\item $F$ and $G$ both restrict to functors $F\colon \ccD\to \ccD'$ and $G\colon \ccD'\to\ccD$, and
\item $F$ detects maps in $\ccD$: for any $f\in\cC$, $F(f)$ is in $\ccD'$ if and only if $f$ is in $\ccD$. 
\end{enumerate}
Then $F$ and $G$ induce an adjoint pair: 
\[
G\colon \cP^{\cC'}_{\ccD'}\rightleftarrows \cP^{\cC}_{\ccD}\colon F.
\]
\end{theorem}

\begin{proof}
We first describe the natural transformations on Hom objects.

If
\[
G(X)\xleftarrow{f} A\xrightarrow{g} B\xrightarrow{h} Y
\]
is in $\cP^{\cC}_{\ccD}$, then we take this to 
\[
X\xleftarrow{f*} F(A)\xrightarrow{F(g)} F(B)\xrightarrow{F(h)} F(Y).
\]
Since $F$ descends to a functor $\ccD\to\ccD'$, this is a morphism of $\cP^{\cC'}_{\ccD'}$.

For the other direction, since the image of $F$ is essentially a sieve, any morphism of the form
\[
X\xleftarrow{f'} A'\xrightarrow{g'} B'\xrightarrow{h'} F(Y)
\]
can be rewritten as
\[
X\xleftarrow{f} F(A) \xrightarrow{F(g)} F(B) \xrightarrow{F(h)} F(Y).
\]
We take this arrow to 
\[
G(X)\xleftarrow{f_{*}} A\xrightarrow{g} B\xrightarrow{h} Y.
\]
Since $g\in\ccD$ if and only if $F(g)\in \ccD'$, this is a morphism in $\cP^{\cC}_{\ccD}$.

These constructions are natural and clearly inverses to each other.
\end{proof}

\begin{remark}
We note that we only require that $F$ and $G$ descend to functors between $\ccD$ and $\ccD'$, not that they give an adjoint pair. In particular, we will see below (Proposition~\ref{prop:IndEssentiallySieve}) that these conditions are satisfied by restriction and induction for certain subcategories of finite $G$-sets, even when these functors are not adjoint.
\end{remark}

\section{Pullback stable subcategories of
  \texorpdfstring{$\Set^G$}{SetG}}\label{sec:PBStableSubcats} 

Our incomplete Tambara functors are controlled by suitable
subcategories of the category $\Set^G$ of finite $G$-sets.  In this
section, we develop the basic properties of pullback stable
subcategories of $\Set^G$ that we will need for our subsequent work.
In particular, we show that there is an equivalence of posets between
the poset of pullback stable subcategories of $\Set^G$ and the poset
of indexing systems; this classification result gives a
``span-theoretic'' explanation for the importance of indexing systems
in our work on $N_\infty$ operads.

\subsection{Basic properties of pullback stable subcategories}

We now restrict attention to pullback stable subcategories of
$\Set^G$.  This ambient category is very well-behaved, and we will see
that simple assumptions give surprisingly strong results.  Many of the
results in this subsection work more generally; for clarity we
restrict ourselves to this basic case, and leave the (easy)
generalizations to the interested reader.  Additionally, motivated by
our study below of incomplete Tambara functors, we restrict attention
to those pullback stable subcategories of $\ccD$ which are symmetric
monoidal subcategories.  (Here recall that the symmetric monoidal
structure on $\Set^G$ is given by the coproduct, disjoint union.)  In
particular, we are assuming that given maps $f \colon X \to Y$ and $g
\colon X' \to Y'$ in $\aD$, then $\aD$ contains the map $X \coprod X'
\to Y \coprod Y'$.  Restricting to this setting has a very surprising
consequence.

\begin{proposition}\label{prop:InitialImpliesMono}
If $\ccD$ is a pullback stable, symmetric monoidal subcategory of
$\Set^G$ that contains $\emptyset\to\ast$, then $\ccD$ contains all
monomorphisms.
\end{proposition}

\begin{proof}
Since $\ccD$ contains the terminal object and is pullback stable, by
Proposition~\ref{prop:basic} we conclude that it is wide.  Next, any
monomorphism $S \to T$ can be written as  
\[
\emptyset\amalg S\to (T-S)\amalg S\cong T.
\]
Since $\ccD$ is pullback stable, for any finite $G$-set $S$, pulling
back the map $\emptyset \to \ast$ along the terminal map $T-S \to \ast$
implies that the initial map $\emptyset \to T-S$ is in $\ccD$.  Using
the fact that $\ccD$ is symmetric monoidal, we now conclude that any
monomorphism is in $\ccD$.
\end{proof}

The pullback stable subcategories $\ccD$ of finite $G$-sets we
consider will have an additional property: they have all finite
coproducts and the coproducts are created in $\Set^G$.  We will refer to this
property by saying that $\ccD$ is a {\em finite coproduct complete
  subcategory} of $\Set^G$.  Note that any a coproduct complete
subcategory of $\Set^G$ is in fact a symmetric monoidal subcategory,
since the coproduct in $\Set^G$ is the symmetric monoidal product.
From the point of view of the resulting Tambara functors, we will see
that this is a very natural condition due to the following simple lemma.

\begin{lemma}
Let $\ccD$ be a pullback stable subcategory of
$\Set^G$. Then the following are equivalent: 
\begin{enumerate}
\item The category $\ccD$ is a wide and finite coproduct complete
  subcategory of $\Set^G$.
\item The category $\ccD$ is a symmetric monoidal subcategory that
  contains the maps $\emptyset\to\ast$ and $\ast\amalg\ast\to\ast$.
\end{enumerate}
\end{lemma}

\begin{proof}
If $\ccD$ is wide and finite coproduct complete as a subcategory of
$\Set^G$, then in particular, $\emptyset\to\ast$ and
$\ast\amalg\ast\to\ast$ are in $\ccD$.  

For the converse, all monomorphisms are in $\ccD$ by
Proposition~\ref{prop:InitialImpliesMono}.  In particular, empty
coproducts are in $\ccD$.  Next, for any finite $G$-set $T$, the
coproduct $T \amalg T$ in $\Set^G$ is the coproduct in $\ccD$.  To
see this, consider the following pullback diagram
\[
\xymatrix{
T \amalg T \ar[r] \ar[d] & T \ar[d] \\
\ast \amalg \ast \ar[r] & \ast.
}
\]
Pullback stability implies that the fold map $T\amalg T\to T$ is in
$\ccD$.  Moreover, the two maps $T \to * \to * \amalg *$ induce the
canonical inclusions $i_1, i_2 \colon T \to T \amalg T$ in $\ccD$.
Moreover, it is clear that this holds for arbitrary finite coproducts
of $T$.  Finally, for an arbitrary finite coproduct, we can write the
required universal map out as a composite of iterated fold maps and
the symmetric monoidal product of maps.  That is, given $T \amalg S$,
the universal map to $Z$ given $T \to Z$ and $S \to Z$ can be
expressed as the composite $S \amalg T \to Z \amalg Z \to Z$.
\end{proof}

Pullback stability itself implies a partial converse to this sort of result.

\begin{proposition}\label{prop:SummandsinD}
Let $\ccD$ be a pullback stable subcategory of $\Set^G$ and assume
that $\coprod_i T_i$ is the coproduct in $\ccD$.  If $f \colon S \to
\coprod_i T_i$ is in $\ccD$, then so are $f|_{S_i}$, where
$S_i=f^{-1}(T_i)$.
\end{proposition}

\begin{proof}
The restrictions to these summands are just the pullbacks along the
inclusions of $T_i$ into the coproduct. 
\end{proof}

In particular, the conditions of being a wide, pullback stable, and
finite coproduct complete subcategory of $\Set^G$ are extremely
stringent: we can completely recover any subcategory of $\Set^G$ of
this form out of a subcategory of the orbit category of $G$. 

\begin{definition}
Let $\ccD$ be a subcategory of $\Set^G$.  We define $\cOrb_{\ccD}$ to
be the full subcategory of $\ccD$ obtained by restricting the objects
to the orbits $G/H$ (for $H\subset G$) that are contained in $\ccD$.
\end{definition}

\begin{remark}
When $\ccD=\Set^G$, this is the ordinary orbit category, so we can rewrite $\cOrb_{\ccD}$ as $\ccD\cap \cOrb$.
\end{remark}

\begin{proposition}\label{prop:CoprodComp}
For $\ccD$ as above, $\ccD$ is the finite coproduct completion of
$\cOrb_{\ccD}$ in $\Set^G$; it is the smallest subcategory of $\Set^G$
containing $\cOrb_{\ccD}$ that has all finite coproducts. 
\end{proposition}

\begin{proof}
Let $f\colon S\to T$ be a map in $\ccD$.  Decomposing $T$ into orbits,
we write $T=\coprod G/H_i$, and let $S_i=f^{-1}(G/H_i)$.  Since $\ccD$ is closed
under pullbacks and is symmetric monoidal, the restriction of $f$ to
$S_i$ is again in $\ccD$ if and only if $f$ itself is by
Proposition~\ref{prop:SummandsinD}. It therefore suffices to consider
$T=G/H$ is an orbit. 

Similarly, decomposing $S$ into orbits we can write $S=\coprod G/H_j$, and since
$\ast\amalg \ast\to\ast$ is in $\ccD$, we see that $f$ is in $\ccD$ if
and only if $f|_{G/H_j}$ is in $\ccD$. In particular, any map in
$\ccD$ is a sum of maps in $\ccD$ between orbits. In particular, $f$
is in $\ccD$ if and only if it is in the finite coproduct completion of
$\cOrb_{\ccD}$. 
\end{proof}

\begin{remark}
This should be viewed as an analogue of the classical result that
$\Set^G$ is the finite coproduct completion of the full orbit category $\cOrb$.
\end{remark}

Proposition~\ref{prop:CoprodComp} provides a conceptual understanding
of what our incomplete Tambara functors look like; for this restricted
class of pullback stable subcategories, everything is determined by
which maps are in $\cOrb_{\ccD}$. 

\subsection{Subcategories of \texorpdfstring{$\Set^{G}$}{SetG} from indexing systems}

In this subsection, we explain how an indexing system determines a
wide, pullback stable, and finite coproduct complete subcategory of
$\Set^G$.  Recall (see Definition~\ref{defn:IndexingSystem}) that an
{\em indexing system} is a full symmetric monoidal 
sub-coefficient system $\mC$ of $\mSet$ that contains all trivial sets
and is closed under finite limits and self-induction, in the sense
that if $H/K\in \mC(H)$ and $T\in\mC(K)$, then
$H\times_{K}T\in\mC(H)$.  (Recall also that we refer to the sets
specified by $\cO$ as the admissible sets of $\cO$.)

\begin{definition}
For an indexing system $\cO$, let $\Set_{\cO}^{G}$ denote the wide
subgraph of $\Set^G$ where $f\colon S\to T$ is in $\Set_{\cO}^G$ if
and only if for all $s\in S$, 
\[
G_{f(s)}/G_{s}\in{\cO}(G_{f(s)}).
\]
\end{definition}

Consideration of orbits immediately gives an equivalent formulation of
the condition for maps to be in $\Set_{\cO}^G$.

\begin{proposition}\label{prop:OrbitReduction}
A map $f\colon S\to T$ is in $\Set_{\cO}^G$ if and only if for all $s\in S$, we have 
\[
 G_{f(s)}\cdot s\in \cO\big( G_{f(s)}\big).
\]
\end{proposition}

The fact that $\cO$ is an indexing system implies that
$\Set_{\cO}^{G}$ is in fact a category.

\begin{theorem}
The graph $\Set_{\cO}^{G}$ forms a category.
\end{theorem}

\begin{proof}
Since trivial sets are admissible for any subgroup of $G$,
$\Set_{\cO}^{G}$ contains the identity map for each object.
Therefore, it suffices to show that given two morphisms $f_1\colon
S_1\to S_2$ and $f_2\colon S_2\to S_3$ in $\Set_{\cO}^{G}$, $f_2\circ
f_1$ is also in $\Set_{\cO}^{G}$.

Let $s\in S_1$ be any element.  By assumption, $G_{f_{1}(s)}/G_{s}\in
{\cO}(G_{f_{1}(s)})$, and $G_{f_{2}\circ f_{1}(s)}/G_{f_{1}(s)}\in
{\cO}(G_{f_{2}\circ f_{1}(s)})$.  Since admissible sets are closed
under self-induction, we conclude that
\[
G_{f_{2}\circ f_{1}(s)}/G_{s}\cong G_{f_{2}\circ
  f_{1}(s)}\times_{G_{f_{1}(s)}} G_{f_{1}(s)}/G_{s}\in
{\cO}(G_{f_{2}\circ f_{1}(s)}),
\]
and therefore the composite is also in the category.
\end{proof}

Since the condition of being a map in $\Set_G^{\cO}$ is determined
orbit-by-orbit and by assumption $\emptyset\in\cO(H)$ for all $H$, the
following is immediate.

\begin{proposition}
For any indexing system $\cO$, the category $\Set^G_{\cO}$ is a finite
coproduct complete subcategory of $\Set^G$.
\end{proposition}

Finally, we show that $\Set^G_{\cO}$ is pullback stable.  For this, we
need to check that the maps in $\Set^G_{\cO}$ are closed under
induction, in the following sense.

\begin{proposition}\label{prop:InductionClosure}
A map $f\colon S\to T$ is in $\Set_{i_H^\ast\cO}^H$ if and only if
\[
G\times_H f\colon G\times_H S\to G\times_H T\in\Set_{\cO}^G.
\]
\end{proposition}

\begin{proof}
We observe that the stabilizer of the points in $G\times_H S$ can be determined by those of $S$:
\[
G_{[(g,s)]}=g G_{s}g^{-1}.
\]

If we let $F=G\times_H f$, then we have an identification
\[
G_{F(g,s)}/G_{(g,s)}=gG_{f(s)}g^{-1}/gG_{s}g^{-1}=g^{\ast} G_{f(s)}/G_s,
\]
where $g^\ast\colon \Set^{G_{s}}\xrightarrow{\cong} \Set^{gG_{s}g^{-1}}$ is the multiplication by $g$ map. In general, $U$ is an admissible $G_{s}$-set if and only if $g^\ast U$ is an admissible $gG_{s}g^{-1}$-set, from which the result follows.
\end{proof}

\begin{theorem}
For any indexing system $\cO$, the subcategory $\Set^G_{\cO}$ is pullback stable.
\end{theorem}

\begin{proof}
Consider a pullback diagram in $\Set^G$
\[
\xymatrix{
{U}\ar[r]^h\ar[d]_k & {S}\ar[d]^f \\
{V}\ar[r]_g & {T,}
}
\]
where the map $f\colon S\to T$ is in $\Set_{\cO}^G$ and where $g$ is
arbitrary.  We will show that $k$ is a map in $\Set_{\cO}^G$.  First,
we reduce to the case that $T=G/G$ is the terminal object.

Consider an element $u\in U$. By assumption, we have that 
\[
H:=G_{f\circ h(u)}=G_{g\circ k(u)},
\]
and by naturality, $H$ contains $G_{u}$, $G_{h(u)}$, and $G_{k(u)}$.  Proposition~\ref{prop:OrbitReduction} shows that it suffices to work $H$-equivariantly, and we need only look at the points which map to $t=g(k(u))=f(h(u))\in T$. Hence we can replace our original diagram with an $H$-equivariant one:
\[
\xymatrix{
{H\cdot k(u)\times H\cdot h(u)}\ar[r]^-h\ar[d]_k & {H\cdot h(u)}\ar[d]^f \\
{H\cdot k(u)}\ar[r]_g & {\{t\}.}
}
\]

We are therefore reduced to showing that if $G/H$ is an admissible $G$-set and if $G/K$ is arbitrary, then $G/K\times G/H\to G/K$ is in $\Set_{\cO}^G$. However, the projection map is isomorphic to the map
\[
G\times_K(i_K^\ast G/H\to K/K).
\]
Since $G/H$ is an admissible $G$-set, $i_K^\ast G/H$ is an admissible $K$-set for any subgroup $K$, and therefore by Proposition~\ref{prop:InductionClosure}, this map is in $\Set_{\cO}^G$.
\end{proof}

\begin{remark}
Roughly speaking, the preceding results imply that $\Set^G_{\cO}$ is
essentially the finite coproduct completion of the subcategory
determined by objects isomorphic to the union of the essential images
of the induction functors $\Set^H\to\Set^G$ restricted to $\cO(H)$.
\end{remark}

We close this subsection with an extremely important observation which
we will need in our study of the change of groups.  Since the action
map is the pullback of $G/H\to *$ along $T\to *$, the following is
immediate.

\begin{corollary}\label{cor:Factorization}
If $G/H$ is admissible for $\cO$, then for any $T$, the action map 
\[
\epsilon_{T}\colon G\times_{H} i_{H}^{\ast} T\to T
\]
is in $\Set^{G}_{\cO}$.
\end{corollary}

\subsection{An intrinsic formulation of indexing systems}

The main result of this subsection is that all wide, pullback stable,
finite coproduct complete subcategories of $\Set^G$ are of the form
$\Set^G_{\cO}$ for some $\cO$.  This characterization provides an
intrinsic description of the data of an indexing system.  We begin
with a trivial observation.

\begin{lemma}\label{lem:AdmissibleInclusion}
If $\cO\subset \cO'$, then we have an inclusion
\[
\Set^{G}_{\cO}\subset \Set^{G}_{\cO'}.
\]
\end{lemma}

As a consequence of Lemma~\ref{lem:AdmissibleInclusion}, $\Set^G_{(-)}$
is a functor from the poset of indexing systems to the poset of wide,
pullback stable, finite coproduct complete subcategories of $\Set^G$.

\begin{theorem}\label{thm:Isomorphism}
The functor $\Set^G_{(-)}$
\[
\cO\mapsto \Set^G_{\cO}
\]
gives an isomorphism between the poset of indexing systems and the
poset of wide, pullback stable, finite coproduct complete
subcategories of $\Set^G$.
\end{theorem}

We prove this by explicitly constructing an inverse to this functor in
a series of lemmas that constitute the remainder of the section.  In
the following discussion, let $\ccD$ be a wide, pullback stable,
finite coproduct complete subcategory of $\Set^G$. 

\begin{lemma}
If $\ccD$ is as above, then there is a coefficient system of
categories $\cO_{\ccD}$ specified at $G/H$ by the assignment 
\[
\cO_{\ccD}(G/H)=\ccD_{/G/H},
\]
where $\ccD_{/G/H}$ denotes the overcategory of $G/H$ in $\ccD$.
\end{lemma}

\begin{proof}
First, the pullback stability of $\ccD$ implies that for any map of
finite $G$-sets $f\colon S\to T$, pullback along $f$ induces a functor 
\[
\ccD_{/T}\to\ccD_{/S}.
\]
Since the restriction maps in $\mSet$ come from pulling back along
maps of orbits, the result follows.
\end{proof}

Next, we show that $\cO_{\ccD}$ is a symmetric monoidal coefficient
system.  First, observe that since the forgetful functor $\ccD_{/T}
\to \ccD$ creates all colimits, the fact that $\ccD$ is finite
coproduct complete in $\Set^G$ implies that the analogous result for
$\ccD_{/T}$ holds.

\begin{lemma}
The slice categories $\ccD_{/T}$ are all finite coproduct complete as
subcategories of $\Set^G_{/T}$.
\end{lemma}

Since the coproduct is the symmetric monoidal product, this has the
following immediate consequences.

\begin{corollary}
For any $T$, the slice category $\ccD_{/T}$ is a symmetric monoidal
subcategory of $\Set^G_{/T}$.
\end{corollary}

\begin{corollary}\label{cor:SliceSymMonoidal}
For $\ccD$ as above, $\cO_{\ccD}$ is a symmetric monoidal coefficient system of $\mSet$.
\end{corollary}

Another consequence of the slice categories $\ccD_{/T}$ being finite
coproduct complete is that when $T$ is an orbit, $\ccD_{/T}$ is a full
subcategory.  This is the only place in this subsection where we use
specific properties of the category $\Set^G$.

\begin{proposition}\label{prop:Full}
The slice category $\ccD_{/G/H}$ is a full subcategory of $\Set^H$.
\end{proposition}

\begin{proof}
Let $T\to G/H$ and $S\to G/H$ be two elements in the overcategory $\ccD_{/G/H}$. These are isomorphic in the overcategory to maps of the form $G\times_H T'\to G/H$ and $G\times_H S'\to G/H$, where $T'$ and $S'$ are the respective corresponding $H$-sets. A map of finite $G$-sets $T\to S$ over $G/H$ is the same data (by the equivalence of categories) as an $H$-map $T'\to S'$. Choosing orbit decompositions of $T'$ and $S'$ shows that any $f\colon T'\to S'$ is isomorphic to one of the form
\[
\coprod_i \amalg f_{i,j}\colon \coprod_i \amalg_j T_{i,j}'\to \coprod_i S_i',
\]
where $T_{i,j}'$ and $S_i'$ are $H$-orbits for all $i$ and $j$. It therefore suffices to show that whenever
 $G/K\to G/H$ and $G/J\to G/H$ are in $\ccD$, then every map of orbits
\[
H/K\to H/J
\]
induces up to a map in $\ccD$. Since any such map is $H$-isomorphic to a canonical quotient $H/K\to H/J$ where $K\subset J$, without loss of generality, we may assume our map of orbits is of this form.

Here is where the properties of finite $G$-sets appear. Consider the pullback 
\[
G/K\times_{G/H} G/J\to G/J
\] 
of the map $G/K\to G/H$ along the map $G/J\to G/H$. By assumption, this map is in $\ccD$. However, $G/K\times_{G/H} G/J\to G/J$ is isomorphic to the obvious map
\[
G\times_J(i_J^\ast H/K)\to G/J.
\]
(The appearance of $H$ here is made more transparent by considering
instead the equivalent $H$-equivariant isomorphism and inducing back
up. This also shows that this isomorphism takes place in the
overcategory of $G/H$.) 
Since $J/K$ is a summand of $i_J^\ast H/K$, $G/K$ is a summand of
$G\times_J i_J^\ast H/K$, and the inclusion of this summand is
automatically in $\ccD$. Composing these two maps shows that $G/K\to
G/J$ is in $\ccD$ as desired.
\end{proof}

\begin{remark}
This is a very surprising asymmetry in the argument here: at no point
did we actually use that $S\to G/H$ was in $\ccD$. This indicates that
maps in $\cOrb_{\ccD}$ are much weirder than expected at first
blush. Thinking of this as the statement ``if $T$ is admissible for
$H$, then its restriction is admissible for any subgroup of $H$''
makes this phenomenon less confusing.
\end{remark}

\begin{lemma}
For any $\ccD$ as above, the symmetric monoidal coefficient system $\cO_{\ccD}$ is an indexing system.
\end{lemma}
\begin{proof}
We verify the conditions from Definition~\ref{defn:IndexingSystem}. Proposition~\ref{prop:Full} shows that this is a full symmetric monoidal sub-coefficient system, so we need only verify that it is closed under finite limits and under self-induction.

To show that $\cO_{\ccD}(H)$ is closed under finite limits, we show it is closed under subobjects and under products. Both of these follow immediately from pullback stability. Proposition~\ref{prop:InitialImpliesMono} shows that if $T_1\amalg T_2\to G/H$ is in $\ccD$, then the restrictions $T_1\to G/H$ and $T_2\to G/H$ are both in $\ccD$ as well, showing closure under subobjects. For products, if $S,T\to G/H$ are in $\ccD$, then the map $S\times_{G/H} T\to S\to G/H$ is a composite of maps in $\ccD$ and hence in $\ccD$.

Closure under self-induction is actually just the statement that
$\ccD$ forms a category, since induction along a map $G/H\to G/K$ is
just post-composition.  Stipulating that $T\to G/K$ and $G/K\to G/H$
are both in $\ccD$ is equivalent to the statements that
$T'\in\cO_{\ccD}(K)$, where $T'$ is the inverse image of $eK$, and
$H/K\in\cO_{\ccD}(H)$.  Then the composite $T\to G/K\to G/H$ is in
$\ccD$, which is equivalent to the fact that $H\times_K
T'\in\cO_{\ccD}(H)$.
\end{proof}

\begin{corollary}\label{cor:invfunc}
The assignment 
\[
\ccD\mapsto \cO_{\ccD}
\]
gives a functor from the poset of wide, pullback stable, coproduct complete subcategories of $\Set^G$ to the poset of indexing systems.
\end{corollary}

To complete the proof of Theorem~\ref{thm:Isomorphism}, it suffices to
show that the functor constructed in Corollary~\ref{cor:invfunc} is an
inverse to $\Set^G_{(-)}$ on objects; since the categories in question
are posets, functors inducing a bijection on objects participate in an
equivalence of categories.  The next two lemmas complete this
verification.

\begin{lemma}
For any $\ccD$ as above, 
\[
\ccD=\Set^G_{\cO_{\ccD}}.
\]
\end{lemma}

\begin{proof}
Since for any indexing system $\cO$, $\Set^G_{\cO}$ is always satisfies the conditions on $\ccD$, Proposition~\ref{prop:CoprodComp} applies. However, $K/H\in\cO_{\ccD}(K)$ if and only if $G/H\to G/K$ is in $\cOrb_{\ccD}$. By definition, $K/H\in\cO_{\ccD}(K)$ if and only if $G/H\to G/K$ is in $\cOrb_{\cO}$, so we see that $\cOrb_{\ccD}=\cOrb_{\cO}$, proving the result.
\end{proof}

\begin{lemma}
For any indexing system $\cO$, 
\[
\cO=\cO_{\Set^G_{\cO}}.
\]
\end{lemma}
\begin{proof}
Consider the slice category over $G/H$ of $\Set^G_{\cO}$. Any object $T\to G/H$ in the slice category over $G/H$ is isomorphic in the slice category to an object of the form $G\times_H T'\to G/H$, where $T'\to \ast$ is the canonical map. By Proposition~\ref{prop:InductionClosure}, $T\to G/H$ is in $\Set^G_{\cO}$ if and only if $T'\to\ast$ is in $\Set^H_{i_H^\ast\cO}$. However, it is immediate from the definition that $T'\to\ast$ is in $\Set^H_{i_H^\ast\cO}$ if and only if $T'\in i_H^\ast \cO(H)=\cO(H)$, proving the desired result.
\end{proof}

\section{Incomplete Tambara Functors}\label{sec:OTambara}

In this section, we use the work of the preceding sections to define
incomplete Tambara functors in terms of indexing systems.
Specifically, we construct a category of incomplete Tambara functors
that corresponds to any indexing system $\cO$ via the functor
$\Set^G_{(-)}$ described above.  

Using Theorem~\ref{thm:PolynomialSubcategories}, we can in fact define
$\ccD$-Tambara functors for any wide, pullback stable, symmetric
monoidal subcategory of $\Set^G$.

\begin{definition}\label{defn:incompleteTamb}
Let $\ccD$ be a wide, pullback stable, symmetric monoidal subcategory
of $\Set^G$.  A $\ccD$-semi-Tambara functor is a product preserving functor 
\[
\cP^{G}_{\ccD}\to\Set.
\]

A $\ccD$-Tambara functor is an $\ccD$-semi-Tambara functor that is [abelian] group valued.

When $\ccD=\Set^{G}_{\cO}$, then we will call $\ccD$-[semi]-Tambara functors simply $\cO$-[semi]-Tambara functors.
\end{definition}

In $\cP^{G}_{\ccD}$, the morphism sets have a natural commutative monoid structure given by disjoint union:
\[
[X\leftarrow S\to T\to Y]+[X\leftarrow S'\to T'\to Y]=[X\leftarrow S\amalg S'\to T\amalg T' \to Y].
\]
Following Tambara, we can therefore group complete this category by
group completing each of these morphism sets \cite{Tambara}.  This
is analogous to the passage from the category of spans of finite
$G$-sets to the Burnside category.  Using Tambara's original argument,
we then obtain a characterization of $\ccD$-Tambara functors in terms
of the group completion.

\begin{proposition}[{\cite{Tambara}}]
For any $\ccD$-semi-Tambara functor $\m{T}$, there is a unique $\ccD$-Tambara structure on the group completion of $\m{T}$.

An $\ccD$-Tambara functor is an additive functor from the group completion of $\cP^{G}_{\ccD}$ to abelian groups.
\end{proposition}

The definition of $\ccD$-Tambara functors in terms of polynomials with
exponents in a wide, pullback stable subcategory of $\Set^G$ makes
proving structural theorems remarkably straightforward.  However, a
priori, it is not clear how to understand the structure on
$\ccD$-Tambara functors in terms of the structure of $\ccD$.  We now
explain how properties of $\ccD$ give rise to familiar structures on
$\ccD$, culminating in a characterization of $\cO$-Tambara functors as
Green functors with additional structure in
Theorem~\ref{thm:GenAndRels}.

We begin by looking at the consequence of the simple observation that
every wide, pullback stable subcategory of $\Set^G$ contains
$\Set^G_{Iso}$, the category of finite $G$-sets and isomorphisms.

\begin{proposition}
A $\Set^{G}_{Iso}$-Tambara functor is a Mackey functor.
\end{proposition}

\begin{proof}
Any bispan of the form 
\[
X\leftarrow S\xrightarrow{\cong} T \to Y
\]
is canonically isomorphic (via the isomorphism $S\to T$) to one of the form
\[
X\leftarrow S\xrightarrow{Id} S\to Y.
\]
The category of such bispans is the ordinary category of spans of $G$-sets.
\end{proof}

This implies that for $\ccD$-Tambara functors are Mackey functors with extra structure.

\begin{corollary}
For any wide, pullback stable subcategory $\ccD$ of $\Set^G$, a $\ccD$-Tambara functor has an underlying Mackey functor.
\end{corollary}

The next simplest case is when $\aD$ is the collection of finite $G$-sets and
monomorphisms.  This situation was analyzed by Hoyer in his thesis, so we just cite the result here.

\begin{proposition}[{\cite{HoyerThesis}}]
A $\Set^G_{Mono}$-Tambara functor is a ``pointed Mackey functor'': a Mackey functor $\mM$ together with a map $\mA\to \mM$.
\end{proposition}

The pointedness of $\Set^G_{Mono}$-Tambara functors arises from the fact that $\Set^G$ has an initial object $\emptyset$, and hence for any $T$, there is a distinguished morphism $\emptyset\to T$ in the category of polynomials with exponents in $\Set^G_{Mono}$, namely
\[
\emptyset\leftarrow\emptyset \to T\xrightarrow{=} T.
\]

The canonical map from the empty set, together with the isomorphisms,
generates all of $\Set^G_{Mono}$ as a symmetric monoidal category. To
build the rest of $\Set^G$, we need to also include the projections
$G/H\to G/K$ for $H\subset K$ and the fold maps $T\amalg T\to T$. 

First, we study the consequences of including the fold map in $\aD$;
the proof of the following proposition is exactly the same as that
given by Tambara, so we omit it.

\begin{proposition}[{\cite[2.3]{Tambara}}]
Let $\ccD$ be a wide, pullback stable subcategory such that for some $T$, the fold map $T\amalg T\to T$ is in $\ccD$. Then the following hold for any $\ccD$-[semi-]Tambara functor $\m{R}$:
\begin{enumerate}
\item For all $S\to T$, $\m{R}(S)$ is a non-unital, commutative [semi-]ring,
\item For all $S_1\xrightarrow{f} S_2\to T$, the restriction map $R_f$ is a non-unital, commutative [semi-]ring map, 
\item For all $S_1\xrightarrow{f} S_2\to T$, the norm map $N_f$ is a map of multiplicative monoids, and
\item For finite $G$-sets over $S_1\xrightarrow{f} S_2\to T$, the
  Frobenius relation holds: 
\[
a\cdot T_f(b)=T_f(R_f(a)\cdot b).
\]
\end{enumerate}
If moreover $\emptyset\to T$ is in $\ccD$, then the all of the results in the previous list are instead for unital [semi-]rings.
\end{proposition}

\begin{corollary}
If $\emptyset \to \ast$ and $\ast\amalg\ast\to\ast$ are in $\ccD$, then any $\ccD$-Tambara functor $\m{R}$, has an underlying Green functor.
\end{corollary}

This last corollary gives us a description of $\cO$-Tambara functors
as enhanced Green functors, which parallels the situation with
$N_\infty$ ring spectra.

\begin{corollary}
Let $\cO$ be any indexing system.  A $\cO$-Tambara functor has an underlying Green functor. 
\end{corollary}

\begin{remark}
Strickland's ``green'' condition on maps in his formulation of Green functors is exactly the condition that the map be in the wide, pullback stable symmetric monoidal subcategory of $\Set^G$ containing $\emptyset\to \ast$ and $\ast\amalg\ast\to\ast$. This is equivalent to being in $\Set^G_{\cO^{tr}}$, where $\cO^{tr}$ is the indexing system of trivial sets.
\end{remark}

At this point, we have almost all of the structure present in a
$\cO$-Tambara functor.  We only need to understand the effect of the
inclusion of the maps $G/H\to G/K$ in $\aD$.  For this, it can be
helpful to recall an alternative formulation to the axioms for the
norms in a Tambara functors (analogous to the Weyl double coset
formulation of the compatibility of transfers and restrictions). This
has been described in detail in work of Mazur, but we reproduce it
here for clarity~\cite[Theorem 3.1]{MazurArxiv}. The compatibility
with addition was proved by Tambara to show that certain formulae
relating transfers and the Evans norm hold universally.

\begin{proposition}[{\cite{Tambara}}, {\cite[Theorem 3.1]{MazurArxiv}}]\label{prop:SumTambaraRec}
There is a universal formula expressing the norm of a sum:
\[
N_H^K(a+b)=T_{f_{s}}\circ N_{g_{s}}\circ R_{h_{s}}(a\oplus b),
\]
where $h_s$ is the composite of
\[
G\times_K\big(K/H\times\Map(K/H,\ast\amalg\ast)\big)\xrightarrow{\Delta\times 1}G\times_K\big(K/H\times K/H\times \Map(K/H,\ast\amalg \ast)\big)
\]
with the evaluation map
\[
G\times_K\big(K/H\times K/H\times \Map(K/H,\ast\amalg \ast)\big)\xrightarrow{1\times eval} G/H\amalg G/H,
\]
where 
\[
g_s\colon G\times_K(K/H\times \Map(K/H,\ast\amalg\ast)\to G\times_K\Map(K/H,\ast\amalg\ast)
\]
is the canonical quotient and where
\[
h_s\colon G\times_K\Map(K/H,\ast\amalg\ast)\to\ast
\]
is the canonical map.

In particular, this formula depends only on $H\subset K$.
\end{proposition}
\begin{proof}
The composite $N_H^K\circ(-+-)$ is the composite $N_{\pi_{K/H}}\circ T_{\nabla}$, where $\pi_{K/H}\colon G/H\to G/K$ is the canonical quotient and $\nabla\colon G/H\amalg G/H\to G/H$ is the fold map. Since the composite in question is the pullback along the map $G/K\to G/G$ of the case where $K=G$, it suffices to consider only this case. Here, it is not difficult to check from the definition that
\[
\xymatrix{
{G/H}\ar[d]_-{\pi_{G/H}} & {G/H\amalg G/H}\ar[l]_-{\nabla} & {G/H\times \Map(G/H,\ast\amalg\ast)}\ar[l]_-{h_a}\ar[d]^-{g_s} \\
{\ast} & & {\Map(G/H,\ast\amalg\ast)}\ar[ll]^-{h_s}
}
\]
is an exponential diagram, and this gives the required formula by Proposition~\ref{prop:NT}.
\end{proof}

We pause here to stress that since we are assuming that our exponents be drawn from a category that is pullback stable, the compatibility of norms with sums is automatically satisfied.

There is a similar description for the norm composed with the transfer.
\begin{proposition}[{\cite[Theorem 3.1]{MazurArxiv}}]\label{prop:TrTambaraRec}
There is a universal formula expressing the norm of a transfer:
\[
N_K^G Tr_H^K(a)= T_{f_t}\circ N_{g_t}\circ R_{h_t}(a),
\]
where
\[
h_t\colon G/K\times\Map_K(G,K/H)\to G/H
\]
is defined by $h_t(gK,\sigma)=g\sigma(g)$, where
\[
g_t\colon G/K\times\Map_K(G,K/H)\to \Map_K(G,K/H)
\]
is the canonical quotient and where
\[
f_t\colon \Map_K(G,K/H)\to\ast
\]
is the unique map.
\end{proposition}

The following is then immediate from the construction.

\begin{proposition}\label{prop:NormsfromD}
If $G/H\to G/K$ is in $\ccD$, then a $\ccD$-Tambara functor $\m{R}$ has a norm map
\[
\m{R}(G/H)\to\m{R}(G/K)
\]
that satisfies the universal formulae specified by Propositions~\ref{prop:SumTambaraRec} and \ref{prop:TrTambaraRec}.

If $G/K\amalg G/K\to G/K$ is in $\ccD$ as well, then this norm is a map of multiplicative monoids, and if $\emptyset\to G/K$ is in $\ccD$, then it is unital.
\end{proposition}

Putting this together, we can give an alternate formulation of a
$\cO$-Tambara functor.  This is an extremely useful characterization, as it
allows us to use the proofs of many results in the literature on
ordinary Tambara functors to deduce results about $\cO$-Tambara functors. 

\begin{theorem}\label{thm:GenAndRels}
Let $\cO$ be an indexing system.  A $\cO$-Tambara functor is a
commutative Green functor $\m{R}$ together with norm maps of
multiplicative monoids 
\[
N_H^K\colon \m{R}(G/H)\to\m{R}(G/K)
\]
for each $G/H\to G/K\in\cOrb_{\cO}$ that satisfy the Tambara
reciprocity relations given in Proposition~\ref{prop:SumTambaraRec} and \ref{prop:TrTambaraRec}.
\end{theorem}
\begin{proof}
Since $\Set^G_{\cO}$ is a wide, pullback stable, coproduct complete subcategory of $\Set^G$, any $\cO$-Tambara functor is a Green functor plus norm maps. Proposition~\ref{prop:NormsfromD} shows that if $H\subset K$ is such that $G/H\to G/K\in\Set^G_{\cO}$, then we have a norm map satisfying the desired properties. Finally, any map in $\Set^G_{\cO}$ can be written as a composite of iterated fold maps and disjoint unions of maps of the form $G/H\to G/K$, so by naturality, to such a composite we associate the corresponding product of norm maps. These steps are clearly reversible, again using that any map in $\Set^G$ and in $\Set^G_{\cO}$ can be written as a coproduct of disjoint unions of maps of orbits.
\end{proof}

As a straightforward corollary of this result, we obtain the following
consistency result connecting $\cO$-Tambara functors to $N_\infty$
ring spectra.

\begin{theorem}\label{thm:pi0}
Let $\cO$ be an $N_\infty$ operad and $R$ an $\cO$-algebra in
orthogonal $G$-spectra.  Then $\pi_0(R)$ is an $\cO$-Tambara functor.
\end{theorem}

We include one final example of an interesting kind of $\ccD$-Tambara
functor.  Consider the category $\Set^G_{epi}$ of finite $G$-sets and
epimorphisms.  This is visibly a wide, pullback stable, and symmetric
monoidal subcategory of $\Set^G$.  Moreover, it contains
$\ast\amalg\ast\to\ast$.  On the other hand, it is notably missing
$\emptyset\to\ast$, which means that it is not one of the categories
we have considered before.  This lets us give what may be the first
elementary definition of a non-unital Tambara functor: 

\begin{definition}
A \defemph{non-unital Tambara functor} is a $\Set^G_{epi}$-Tambara functor.
\end{definition} 
\begin{proposition}
A non-unital Tambara functor $\m{R}$ is a non-unital, commutative Green functor $\m{R}$ together with maps of multiplicative monoids
\[
N_H^K\colon\m{R}(G/H)\to\m{R}(G/K)
\]
for all $G/H\to G/K$ that satisfy the relations of Proposition~\ref{prop:SumTambaraRec} and~\ref{prop:TrTambaraRec}.
\end{proposition}

\section{Categorical properties of incomplete Tambara functors}\label{sec:OTambprop}

In this section, we describe formal properties of the category of
$\cO$-Tambara functors.  We begin by describing limits and colimits in
$\cO$-Tambara functors.  We then turn to a study of ``change''
functors associated to changing the indexing system.  Finally, we
conclude with discussions of ideals of $\cO$-Tambara functors and
localization phenomena.

\subsection{Limits and colimits in \texorpdfstring{$\cO$}{O}-Tambara functors}

Since $\ccD$-Tambara functors are simply product-preserving functors
into a complete and cocomplete category, the category of all such
functors is clearly complete.

\begin{proposition}
The category of $\ccD$-Tambara functors has all limits, where
\[
\big(\!\lim \m{R}_{i}\big)(T)=\lim \big(\m{R}_{i}(T)\big).
\]
\end{proposition}

Moreover, since limits commute with filtered colimits in this setting,
we immediately deduce the existence of filtered colimits.

\begin{proposition}
The category of $\ccD$-Tambara functors has filtered colimits which
are formed object-wise.
\end{proposition}

The case of arbitrary colimits is much more subtle, and it depends
very heavily only which (if any) fold maps are in $\ccD$. The case
that $\ccD=\Set^{G}_{Iso}$ is classical, and here colimits are also
formed objectwise.  We restrict attention to $\cO$-Tambara functors
from now on. The following is immediate from work of Strickland.

\begin{theorem}
The category of $\cO$-Tamabra functors is cocomplete.  The box product is the coproduct of $\cO$-Tambara functors.
\end{theorem}

\begin{proof}
For coproducts, Strickland (following unpublished work of Tambara) shows that there is a canonical way to define norms on the box product of two Tambara functors in a way that is compatible with the norms on the factors \cite[Proposition 9.1]{Strickland}. The proof proceeds by constructing explicit norm maps and verifying that they satisfy the appropriate relations. Thus, by Theorem~\ref{thm:GenAndRels}, the proof goes through without change for our restricted class of norm maps, since all of the consistency relations also take place in that category.

Next, Strickland's argument for \cite[Propositions 10.5, 10.6]{Strickland} makes no reference of the forms of the polynomials, and hence holds in general to show that $\cO$-Tambara functors have coequalizers.  Since $\cO$-Tambara functors have infinite coproducts constructed as filtered colimits of finite coproducts, the result now follows.
\end{proof}

Since [semi-]$\cO$-Tambara functors are a diagram category, there are enough ``free'' objects. In particular, we can form a resolution of any $\cO$-Tambara functor by particularly simply ones, which allows for more direct computation. However, these are not immediately amenable to homological algebra constructions, as many of these fail to be flat as Mackey functors (see Warning~\ref{warn:Projective}). Performing an analysis similar to the passage from a rigid $\cO$-Tambara functor to a more homotopical ``$\Ninfty$-algebra in Mackey functors'' fixes this, but we will not focus on that in this paper.

\begin{definition}\label{defn:Frees}
If $H\subset G$, let
\[
\mA^{\cO}[x_{H}]=\cP^{G}_{\cO}(G/H,-)
\]
be the $\cO$-Tambara functor represented by $G/H$.
\end{definition}
The notation here is chosen to draw attention to the parallels with ordinary free commutative rings. 

By the Yoneda lemma, we have a natural isomorphism
\[
\OTamb(\mA^{\cO}[x_{H}],\m{R})\cong \m{R}(G/H),
\]
so in particular, given any $\cO$-Tambara functor $\m{R}$, we can find a free Tambara functor of the form
\[
\mA^{\cO}[x_{H_{1}},\dots]:=\mA^{\cO}[x_{H_{1}}]\Box\dots
\]
which maps surjectively onto $\m{R}$. In fact, by taking the generating set to be all of $\m{R}(G/H)$ as $H$ varies, we can produce this functorially in $\m{R}$. This allows us to form simplicial resolutions of any $\cO$-Tambara functor by frees.

\begin{warning}\label{warn:Projective}
If $\cO$ is non-trivial, then in general, the underlying Mackey functors for $\mA^{\cO}[x_{H}]$ will not be projective. An illuminating example is given by $G=C_{2}$ and $\cO$ the complete coefficient system. Then we can describe $\mA^{\cO}[x_{G}]=\mA^{\cO}[x]$ as
\[
\xymatrix{
{\mA^{\cO}[x](G/G)}\ar@(r,r)[d]^{res} & & {\Z[t]/(t^{2}-2t)[x,nx]/t(nx-x^{2})}\ar@(r,r)[d] \\
{\mA^{\cO}[x](G/e)}\ar@(l,l)[u]^{tr} & & {\Z[x],}\ar@(l,l)[u]
}
\]
where the restriction map takes $nx$ to $x^{2}$ and is the identity on $x$. The transfer map is just multiplication by $t$. (The norm map is induced by $x\mapsto nx$, together with the Tambara relations). As a Mackey functor, this can be rewritten as
\[
\bigoplus_{n\in \N}\mA\oplus \bigoplus_{j\in J}\m{I},
\]
where $\m{I}$ is the augmentation ideal of $\mA$ and where $J$ is a $\Z$-basis for the ideal generated by $nx-x^{2}$ in $\Z[nx,x]$. This has infinite homological dimension.
\end{warning}

\subsection{\texorpdfstring{$\cO$}{O}-Ideals}
Just as in the commutative ring and classical Tambara cases, the kernel of a map between $\cO$-Tambara functors has extra structure. We can define an $\cO$-Tambara ideal in an $\cO$-Tambara functor, generalizing work of Nakaoka \cite{NakaokaIdeals}.

\begin{definition}
If $\m{R}$ is a $\cO$-Tambara functor, then an $\cO$-ideal is a sub-Mackey functor $\m{J}$ such that
\begin{enumerate}
\item The multiplication on $\m{R}$ makes $\m{J}$ an $\m{R}$-bimodule and
\item If $f\colon S\to T$ is in $\Set^{G}_{\cO}$ and is surjective, then $\m{J}$ is closed under $N_{f}$.
\end{enumerate}
\end{definition}

\begin{remark}
At first blush, the surjective condition is somewhat weird. When one recalls that the norm associated to the unique map $\emptyset\to T$ is the multiplicative unit $1$ in $\m{R}(T)$, however, then we see that by excluding this map from our possible maps, we are simply not requiring that $\m{J}(T)$ contain $1$.
\end{remark}

\begin{example}
If $\cO$ is the trivial indexing system, then an $\cO$-ideal is simply the obvious notion of an ideal in a Green functor.
\end{example}

\begin{example}
If $\cO$ is the complete indexing system, then an $\cO$-ideal is a Tambara ideal in the sense of Nakaoka \cite{NakaokaIdeals}.
\end{example}

\begin{example}
If $\cO$ is any indexing system, then a $\cO$-ideal is simply an ideal in the underlying Green functor which is closed under all norms maps indexed by elements in $\cOrb^{\cO}$. In other words, it is an ideal in the underlying Green functor which is simultaneously a sub-non-unital Tambara functor.
\end{example}

Tambara ideals have the feature that the quotient by them is automatically an $\cO$-Tambara functor. The proof is identical to Nakaoka's, so we omit it.

\begin{proposition}[See {\cite[Prop 2.6]{NakaokaIdeals}}]
If $\m{J}$ is an $\cO$-Tambara ideal of $\m{R}$, then $\m{R}/\m{J}$ has an $\cO$-Tambara functor structure such that the natural map $\m{R}\to \m{R}/\m{I}$ is a map of $\cO$-Tambara functors.
\end{proposition}

Given any collection of subsets of $\m{R}(T)$ as $T$ varies, we have a smallest $\cO$-ideal containing them. Informally, it is the closure of these sets under all sums, products, restrictions, transfers, and norms. Some of the most naturally occurring subsets also result in the most pathological $\cO$-ideals, namely those subsets which are simply the entirety of $\m{R}(T)$ for some collection of $G$-sets $T$.

\begin{definition}
If $\cF$ is a family of subgroups of $G$, then let $\m{I}_{\mathcal F}^{\cO}$ be the $\cO$-ideal of the Burnside ring generated by $\mA(G/H)$ for $H\in\cF$.
\end{definition}

\begin{example}
If $\cO$ is the trivial indexing system, then $\m{I}_{\mathcal F}^{\cO}(G/K)$ is the subgroup of $\mA(G/K)$ generated by elements the form $K/H$ where $H\in \cF$.  
\end{example}

The quotients $\m{A}/\m{I}_{\mathcal F}^{\cO}$ show up as various left adjoints to the forgetful functor applied to the zero Tambara functor, as described in Proposition~\ref{prop:NormofZero} below. Similarly, just as $\m{I}_{\mathcal F}^{\cO^{tr}}$ is $\m{\pi}_{0}(E\mathcal F_{+})$, when $\cO$ is the indexing system associated to an $\Ninfty$ operad $\cO$, the rings $\m{A}/\m{I}_{\mathcal F}^{\cO}$ is $\m{\pi}_0$ of the nullification functor killing all cells induced up from elements in the family applied to the zero sphere in the category of $\cO$-algebras.

\subsection{Change of Structure}

If $\ccD\subset \ccD'$, then applying Corollary~\ref{cor:PolynomialInclusion} gives us an inclusion
\[
\cP^{G}_{\ccD}\subset\cP^{G}_{\ccD'}.
\]
We apply this in the case of $\cO\subset \cO'$ and to the obvious inclusion  $\Set^{G}_{Iso}\subset \Set^{G}_{\cO}$. 
These gives us restriction functors.
\begin{proposition}
If $\cO\subset\cO'$, then there is a canonical forgetful functor
\[
\cO'\mhyphen\Tamb\to\OTamb
\]
given by precomposition with the inclusion. This is strong symmetric monoidal.
\end{proposition}
\begin{proof}
The only statement requiring proof is that this is strong symmetric monoidal. Here, though, we simply observe that the symmetric monoidal product is simply the box product on the underlying Mackey functors.
\end{proof}

\begin{remark}
From the point of view of structure on the resulting underlying Green functor, the forgetful functor simply forgets all norms in $\cO'$ that are not also norms in $\cO$.
\end{remark}

\begin{proposition}
For any $\cO$, there is a canonical, strong symmetric monoidal forgetful functor
\[
U\colon\OTamb\to\Mackey.
\]
\end{proposition}

Since the restriction is given by the inclusion of a subcategory, the left adjoint to it is easy to determine.

\begin{proposition}\label{prop:extend}
If $\cO\subset\cO'$, then the restriction
\[
\cO'\mhyphen\Tamb\to \OTamb
\]
has a left adjoint given by Kan extension along the inclusion
$\cP^{G}_{\cO}\subset \cP^{G}_{\cO'}$, which we write as
\[
\cO'\otimes_{\cO}(-)\colon\OTamb\to\cO'\mhyphen\Tamb.
\]
\end{proposition}

The functor described in Proposition~\ref{prop:extend} has a simple
conceptual description.  We freely adjoin norms corresponding to maps
in $\cO$ which are not in $\cO'$ and all of their transfers, and then
we impose relations reflecting the norms being multiplicative
homomorphisms, norms factoring through the Weyl invariants, and the
universal ``Tambara reciprocity'' formulae reviewed in
Propositions~\ref{prop:SumTambaraRec} and \ref{prop:TrTambaraRec}. 

\begin{corollary}
If $\cO$ is any $\Ninfty$ operad, then the forgetful functor
\[
U\colon \OTamb\to\Mackey
\]
has a left adjoint $\Sym_{\cO}(-)$ given by left Kan extension along the inclusion $\cP^{G}_{Iso}\subset \cP^{G}_{\cO}$.
\end{corollary}

Since $\Sym_{\cO}(-)$ is the left Kan extension, we know exactly what
it does to representable Mackey functors, the projective generators of
the category of Mackey functors.  Following the notation of
Definition~\ref{defn:Frees}, we let 
\[
\mA\cdot\{x_{H}\}
\]
denote the Mackey functor represented by $G/H$. We then have a natural isomorphism
\[
\Sym_{\cO}(\mA\cdot\{x_{H}\})\cong \mA[x_{H}].
\]

This makes it very easy to compute $\Sym_{\cO}(\mM)$ for any Mackey functor $\mM$. 
\begin{proposition}
If $\m{P}_{1}\to \m{P}_{0}\to \mM$ is the start of a projective resolution of $\mM$, then $\Sym_{\cO}$ induces an isomorphism
\[
\Sym_{\cO}(\mM)\cong \Sym_{\cO}(\m{P}_{0})\Box_{\Sym_{\cO}(\m{P}_{1})} \mA.
\]
\end{proposition}

\begin{remark}
Here the observation in Warning~\ref{warn:Projective} comes into play: the underlying Mackey functor for a free $\cO$-Tambara functor is essentially never projective. This means care must be taken when performing homological algebra constructions. These results should be seen as the algebra incarnation of the topological result that the $G$-spectrum underlying an equivariant commutative ring spectrum is almost never cofibrant.
\end{remark}

\subsection{Localization of \texorpdfstring{$\cO$}{O}-Tambara functors}

The free $\cO$-Tambara functors of Definition~\ref{defn:Frees} allow us to invert arbitrary collections of elements in $\cO$-Tambara functors. 

\begin{definition}
Let $\m{R}$ be a $\cO$-Tambara functor and let $\m{S}=\{(a_i,T_i) | a_i\in \m{R}(T_i), i\in I\}$ be a collection of elements in the values of $\m{R}$ at various finite $G$-sets. Then we say that a map $\phi\colon\m{R}\to\m{B}$ of $\cO$-Tambara functors \defemph{inverts $\m{S}$} if for all $i\in I$
\[
\phi(a_i)\in\m{B}(T_i)^\times.
\]
\end{definition}

It is clear that if $\phi\colon\m{R}\to\m{B}$ inverts $\m{S}$ and if $\psi\colon\m{B}\to\m{B'}$ is any map of $\cO$-Tambara functors, then $\psi\circ\phi$ inverts $\m{S}$. Thus the subgraph of the category of $\cO$-Tambara functors under $\m{R}$ is a subcategory provided it is non-empty. Luckily, the terminal $\cO$-Tambara functor provides an example.

\begin{proposition}
The zero $\cO$-algebra inverts any set $\m{S}$ for any $\m{R}$.
\end{proposition}

\begin{theorem}\label{thm:InvertingElements}
Let $\mR$ be a $\cO$-Tambara functor and let $\m{S}=\{(a_i,T_i) | i\in I, a_i\in\m{R}(T_i)\}$ be a collection of elements in the values of $\m{R}$ at various finite $G$-sets. Then the category of maps $\phi\colon\m{R}\to\m{B}$  of $\cO$-Tambara functors which invert $\m{S}$ has an initial object.
\end{theorem}
\begin{proof}
If the cardinality of $I$ is infinite, then we simply consider the directed set of finite subsets of $I$ and form the colimit over this. It therefore suffices to show this if $|I|<\infty$. By induction on $|I|$, it therefore suffices to show that we can invert a single element $a\in\m{R}(T)$. 

Consider the $\cO$-Tambara functor 
\[
\m{R}^{\cO}[x_T]:=\m{R}\Box \mA^\cO[x_T].
\]
In the category of $\cO$-Tambara functors under $\m{R}$, this represents the functor which takes $\m{B}$ to $\m{B}(T)$. Although in general the value of the box product on a finite $G$-set is very difficult to understand, we need only describe several expected elements. The unit of the norm-forget adjunction between $\cO$-Tambara functors and Mackey functors has the form
\[
\mA_{T}\hookrightarrow \mA^\cO[x_T],
\]
and the unit of this $\cO$-Tambara functor is a map $\mA\hookrightarrow \mA^\cO[x_T]$. Together, this gives a map
\[
\mA\oplus \mA_{T}\hookrightarrow \mA^{\cO}[x_T].
\]
The reader should this of this as the inclusion of the unit and the degree $1$-monomials. In particular, there is a canonical element $x_T$ in $\mA^{\cO}[x_T](T)$ given by the span $T\leftarrow T\to T$ in $\mA_{T}(T)$. The Yoneda lemma says that any map out of $\mA^{\cO}[x_T]$ is completely determined by its value on this element.

Boxing $\mA^{\cO}[x_T]$ with $\m{R}$ and evaluating at $T$ then gives us an element $ax\in \m{R}^{\cO}[x_T]$, and this choice of element gives us a map of $\cO$-Tambara functors
\[
\mR^{\cO}[y_T]\to\mR^{\cO}[x_T].
\]
At this point, the construction is standard. Let $a^{-1}\m{R}$ be the pushout in $\cO$-Tambara functors under $\m{R}$
\[
\xymatrix{
{\mR^{\cO}[y_T]}\ar[d]\ar[r] &  {\mR^{\cO}[x_T]}\ar[d] \\
{\mR}\ar[r] & {a^{-1}\mR,}
}
\]
where the map from $\mR^{\cO}[y_T]\to \mR$ is the $\mR$-algebra map adjoint to the element $1$ in $\mR(T)$. By construction, a map from $a^{-1}\mR$ to a $\cO$-Tambara functor $\m{B}$ under $\m{R}$ is an element $b\in\m{B}(T)$ such that $\phi(a)b=1\in\m{B}(T)$. Thus $a^{-1}\m{R}$ satisfies the named universal property.
\end{proof}

Inverting an element in a $\cO$-Tambara functor can be an extremely weird operation. For example, it can produce the zero ring for frustratingly many examples.
\begin{example}
If $a\in \m{I}\subset \m{A}$ is any element in the augmentation ideal of the Burnside ring, then the localized Tambara functor $a^{-1}\m{A}$ is always zero.
\end{example}

Just as topologically, it can also be difficult to know whether the localization of $\m{R}$ in the category of $\m{R}$ modules is the same as the localization described above for the category of $\m{R}$-algebras. Consider a set $\m{S}$ as above. If for each $i\in I$, $T_i$ has a trivial $G$-action, then we can copy the arguments of the second author and Hopkins and of Nakaoka to show the following \cite{HHInversion}.

\begin{theorem}\label{thm:InvertingColimit}
Let $a\in\m{R}(G/G)$.  If for all $\pi\colon G/H\to G/K\in\cOrb_{\cO}$, the element $N_{\pi}\circ Res_K^G(a)$ divides a power of $Res_H^G(a)$, then the ordinary sequential colimit
\[
\m{R}\xrightarrow{a\cdot} \m{R}\xrightarrow{a\cdot}\dots
\]
compute the localization $a^{-1}\m{R}$.
\end{theorem}

To make a similar statement for inverting elements in $\m{R}(T)$ for $T$ not a trivial $G$-set, we must consider the various change-of-group functors relating a $\cO$-Tambara functor and an $i_H^\ast\cO$-Tambara functor.

\section{Change functors}\label{sec:ChangeofGroup}

In this section we study the adjunction induced on categories of
$\cO$-Tambara functors associated to a group homomorphism $H \to G$;
we are most interested in the situation where $H \subseteq G$ is a
subgroup.  As one would expect, the situation is precisely analogous
to the situation for $\cO$-algebras in spectra; there is a
``norm-forget'' adjunction involving the admissible sets specified by
$\cO$ (see Proposition~\ref{prop:NTaLeftAdjoint} below).  The
structure we describe here is in fact similarly an aspect of an
incomplete $G$-symmetric monoidal structure on Mackey functors, under
which the commutative monoids are precisely the $\cO$-Tambara
functors.  We intend to describe this structure in detail in a
subsequent paper.

\subsection{Change of Groups}
Observe that induction on $\Set^{H}$ gives a faithful embedding of $\Set^{H}$ into $\Set^{G}$, and thus gives us a faithful embedding
\[
\Ind_{H}^{G}\colon\cP^{H}_{i_{H}^{\ast}\cO}\hookrightarrow \cP^{G}_{\cO}.
\]
Since the product in polynomials is the disjoint union and since induction is strong symmetric monoidal for disjoint unions, this is a product preserving embedding. The following is then immediate.
\begin{proposition}
Precomposing with $\ind_{H}^{G}$ gives the restriction functor
\[
i_{H}^{\ast}\colon \OTamb_{G}\to i_{H}^{\ast}\OTamb_{H}.
\]
\end{proposition}

This functor always has a right adjoint.  For this, we need to apply
Theorem~\ref{thm:AdjointPair}.  This requires a very basic analysis of
the images of the restriction and induction functors when applied to
our categories $\Set^G_{\cO}$.

\begin{proposition}\label{prop:IndEssentiallySieve}
For any $\cO$, the image of $\Ind_{H}^{G}$ restricted to $\Set^{H}_{i_{H}^{\ast}\cO}$ is essentially a sieve in $\Set^{G}_{\cO}$.
\end{proposition}
\begin{proof}
Proposition~\ref{prop:IndEssentiallyaSieve} shows that in $\Set^{G}$, the image of $\ind_{H}^{G}$ is essentially a sieve. Since any map isomorphic to a map in $\Set^{G}_{\cO}$ is in $\Set^{G}_{\cO}$, by Proposition~\ref{prop:InductionClosure}, we are done.
\end{proof}

We have the analogous result for the restriction functors.

\begin{proposition}\label{prop:RestrictionClosure}
The restriction functor $i_{H}^{\ast}\colon \Set^{G}\to\Set^{H}$ restricts to give a functor $\Set^{G}_{\cO}\to\Set^{H}_{i_{H}^{\ast}\cO}$.
\end{proposition}
\begin{proof}
Let $f\colon S\to T$ be a map in $\Set^{G}_{\cO}$, and let $s\in S$. When we consider the $H$-orbit of $s$, the stabilizer of $s$ in $H$ is $H\cap G_{s}$, and similarly for $f(s)$. We therefore need to show that if $G_{f(s)}\cdot s$ is an admissible $G_{f(s)}$-set, then $(H\cap G_{f(s)})\cdot s$ is an admissible $H\cap G_{f(s)}$-set. However, if we consider the restriction of $G_{f(s)}\cdot s$ to $H\cap G_{f(s)}$, then $(H\cap G_{f(s)})\cdot s$ is visibly a disjoint summand. Since the restriction of admissible sets are admissible and since summands of admissible sets are admissible, we conclude that $(H\cap G_{f(s)})$ is admissible.
\end{proof}

\begin{theorem}\label{thm:RightAdjoint}
For all subgroups $H$ and for all indexing systems $\cO$, the functor $i_{K}^{\ast}$ has a right adjoint, $\CoInd_{H}^{G}$ given by
\[
\CoInd_{H}^{G}\m{R}(T):=R(i_{H}^{\ast}T).
\]
\end{theorem}
\begin{proof}
We have an adjoint pair $\Ind_{H}^{G} \leftadjoint i_{H}^{\ast}$ on the category of finite $G$-sets. By Proposition~\ref{prop:IndEssentiallySieve}, we know that the image of the left adjoint $\Ind_{H}^{G}$ is essentially a sieve, so by Theorem~\ref{thm:AdjointPair}, we know that we have an induced adjoint pair $i_{H}^{\ast}\leftadjoint \Ind_{H}^{G}$ on polynomials:
\[
i_{H}^{\ast}\colon \cP^{G}_{\cO}\rightleftarrows \cP^{H}_{i_{H}^{\ast}\cO}\colon \Ind_{H}^{G}.
\]
Since both $i_{H}^{\ast}$ and $\Ind_{H}^{G}$ are product preserving functors, the result follows.
\end{proof}

It is obvious that the functor $\CoInd_{H}^{G}\m{R}$ is a Green functor. The somewhat surprising part of Theorem~\ref{thm:RightAdjoint} is that we have norm maps. These are built in the most na\"{i}ve way possible: simply multiplying together copies. The example of $H=\{e\}$ shows this quite transparently, as we can build on the representation theory story implicit in the underlying Mackey functor.

\begin{example}
If $R$ is a commutative ring (which is an $\cO$-Tambara functor for the trivial group for any $\cO$), then 
\[
\CoInd_{H}^{G} R(T)=Map(T,R)
\]
with the coordinatewise addition and multiplication. Just as the transfer maps are ``sum over cosets'', the norm maps are ``multiply over cosets'':
\[
N_{H}^{G}(f)=\prod_{gH\in G/H} f(gH).
\]
\end{example}

This example gives the intuition for the general case: the new norms are just products over the orbits of the old ones. In other words, the heuristic is that when we consider a norm $N_{\ind f}$ arising by inducing up a map $f$ of $H$-sets, then we just multiply:
\[
N_{\ind f}=\prod_{gH\in G/H} g N_{f}.
\]

The composite $\ind_{H}^{G}\circ i_{H}^{\ast}$ is also readily determined, and this lets us also generalize this functor.
\begin{proposition}\label{prop:KanExtendGmodH}
The composite $\ind_{H}^{G}\circ i_{H}^{\ast}$ is isomorphic to the functor
\[
m_{G/H}(\m{R}):=\m{R}_{G/H},
\]
where $\m{R}_{G/H}(T)=\m{R}(G/H\times T)$.
\end{proposition} 

\begin{corollary}
For any indexing system $\cO$ and for any finite $G$-set $T$, the assignment
\[
m_{T}(\m{R}):=\m{R}_{T}
\]
gives an endofunctor of the category of $\cO$-Tambara functors.
\end{corollary}
\begin{proof}
Since $\m{R}_{T_{1}\amalg T_{2}}\cong \m{R}_{T_{1}}\times\m{R}_{T_{2}}$ and since the category of $\cO$-Tambara functors is closed under limits, it suffices to show this for $T=G/H$. This is Proposition~\ref{prop:KanExtendGmodH}.
\end{proof}

The functor $m_{T}$ can also be described as the Kan extension along the functor on polynomials given by $T\times -$. For formal reasons, its left adjoint should be the functor that is Kan extension along the internal Hom object $F(T,-)$. Since $\m{R}\mapsto \m{R}_{G/H}$ is the composite $\CoInd_{H}^{G}\circ i_{H}^{\ast}$, then the left adjoint will be the composite of $i_{H}^{\ast}$ with its left adjoint. 

We now wish to build a left adjoint to the forgetful functor. Formally, such a left adjoint will be given by the left Kan extension along the inclusion, and work of Kelly-Lack shows that the left Kan extension of any product preserving functor along $\ind_{H}^{G}$ is again product preserving \cite[Prop 2.5]{KellyLack}. This shows the well-definedness of the following definition.

\begin{definition}
Let 
\[
n_{H}^{G}\colon i_{H}^{\ast}\OTamb_{H}\to \OTamb_{G}
\]
be the left Kan extension along the inclusion $\cP^{H}_{i_{H}^{\ast}\cO}\hookrightarrow \cP^{G}_{\cO}$. 
\end{definition}

By the universal property of the left Kan extension, the following is immediate.

\begin{proposition}
The functor $n_{H}^{G}$ is the left-adjoint to the restriction functor $i_{H}^{\ast}$.
\end{proposition}

In general, this is a very difficult functor to understand. We pause here to give a short example that shows how pathological this can be. Let $\m{0}$ be the zero Green functor. This has a unique $\cO$-Tambara functor structure for any $\cO$.

\begin{proposition}\label{prop:NormofZero}
For any $H\subset G$ and for any $\cO$, we have
\[
n_{H}^{G}\m{0}=\mA/\m{I}_{\cF(H)}^{\cO},
\] 
where $\cF(H)$ is the family of subgroups of $G$ subconjugate to $H$.
\end{proposition}

\begin{proof}
Both sides have the same universal property: the space of maps out of them is either a point or empty, with the latter occurring exactly when the restriction to $H$ of the target is non-zero.
\end{proof}

We can use the functor $n_H^G$ to give another formulation of inverting classes, building a result analogous to the colimit formulation of inversion.

\begin{proposition}\label{prop:InvertinH}
If $\m{R}$ is an $\cO$-Tambara functor and $a\in\m{R}(G/H)$, then $a^{-1}\m{R}$ is isomorphic to the pushout in $\cO$-Tambara functors of the diagram
\[
\xymatrix{
{n_H^G i_H^\ast \m{R}}\ar[r]^{\epsilon}\ar[d]_{n_H^G\iota} & {\m{R.}} \\
{n_H^G \big(a^{-1}i_H^\ast \m{R}\big)}
}
\]
\end{proposition}
\begin{proof}
Both the pushout and the localization of $\m{R}$ have the same universal property.
\end{proof}

This allows us to refine Theorem~\ref{thm:InvertingColimit}.
\begin{theorem}
Let $a\in\m{R}(G/H)$. If for all $\pi\colon H/K\to H/J\in \cOrb_{i_H^\ast\cO}$, the element $N_{\pi}\circ\res_K^H(a)$ divides a power of $\res_J^H(a)$, then
\begin{enumerate}
\item the sequential colimit $\m{R}'$
\[
i_H^\ast \m{R}\xrightarrow{a\cdot} i_H^\ast\m{R}\xrightarrow{a\cdot}\dots
\]
computes the localization $a^{-1}i_H^\ast\m{R}$, and
\item $a^{-1}\m{R}$ is the pushout of the diagram
\[
\xymatrix{
{n_H^G i_H^\ast\m{R}}\ar[r]^{\epsilon}\ar[d]_{n_H^G\iota} & {\m{R}.} \\
{n_H^G \m{R}'}
}
\]
\end{enumerate}
\end{theorem}
\begin{proof}
The first part is simply a restatement of Theorem~\ref{thm:InvertingColimit}. The second part follows from the first from Proposition~\ref{prop:InvertinH}.
\end{proof}

\subsection{Identifying \texorpdfstring{$n_{H}^{G}$}{n}}

In general, identifying the Mackey functor underlying $n_{H}^{G}\m{R}$ is difficult. We can single out, however, distinguished cases where such an identification is possible. We being with a construction in Mackey functors, due to Mazur for cyclic $p$-groups and Hoyer for all finite groups.

\begin{definition}[{\cite[Def 2.3.2]{HoyerThesis}}]
Let $N_{H}^{G}\colon \Mackey_{H}\to\Mackey_{G}$ be the left Kan extension along the coinduction map
\[
\Set^{H}\to\Set^{G}.
\]
\end{definition}

\begin{remark}
Since the coinduction functor $F_H(G,-)$ is a strong symmetric monoidal functor for the cartesian product, this definition also gives a functor
\[
N_H^G\colon\Green_H\to\Green_G.
\]
We will therefore blur the distinction between which functor we mean at will.
\end{remark}

\begin{theorem}\label{thm:IdentifyingLeftAdjoint}
If $G/H$ is admissible for $\cO$, then we have a natural isomorphism
\[
U\circ n_{H}^{G}(-)\cong N_{H}^{G}\circ U(-).
\]
\end{theorem}

\begin{proof}
Our proof closely follow's Hoyer's proof for the case of ordinary Tambara functors. We include the details, since there is a single point where care must be taken.

Since $n_{H}^{G}(\m{T})$ is built as a left Kan extension, we can think of an element as an equivalence class of pairs
\[
(T_{h}\circ N_{g}\circ R_{f},m)\in \cP^{G}_{\cO}(G\times_{H} X,Y)\times \m{T}(X).
\]
The coend equivalence relations say that whenever a map like $N_{g}$ is induced up from $H$, then we can move it across:
\[
(T_{\ind f'}\circ N_{\ind g'}\circ R_{\ind h'},m)=\big(1,T_{f'}\circ N_{g'}\circ R_{h'}(m)\big).
\]
We use these to bring the bispan
\[
G\times_{H} X\xleftarrow{f} S\xrightarrow{g} T\xrightarrow{h} Y
\]
into a more canonical form, giving a better spanning set for our coend.

Since any map to an induced object is isomorphic to an induced map, we can replace our bispan with an equivalent one:
\[
G\times_{H} X\xleftarrow{\ind f} G\times_{H} S'\xrightarrow{\epsilon_{T}\circ \ind g'} T\xrightarrow{h} Y,
\]
where $\epsilon_{T}$ is the counit of the adjunction. Since $G/H$ is admissible, Corollary~\ref{cor:Factorization} shows that $\epsilon_T$ is in $\Set^G_{\cO}$, and in particular, the factorization $\epsilon\circ \ind g'$ takes place in the category $\Set_{\cO}^{G}$. This is the only point in Hoyer's argument where the fact that we were in $\Set^G_{\cO}$ arises. Now, any element
\[
(T_{h}\circ N_{g}\circ R_{f},m)\in \cP^{G}_{\cO}(G\times_{H} X,Y)\times \m{T}(X).
\]
is equivalent to one of the form
\[
(T_{h}\circ N_{\epsilon}, N_{g'}\circ R_{f'}m)\in \cP^{G}_{\cO}(G\times_{H} i_{H}^{\ast} T,Y)\times \m{T}(i_{H}^{\ast}T).
\]
This shows that pairs $(T_{h}\circ N_{\epsilon},m)$ for a spanning set for the coend defining $n_{H}^{G}\m{T}$.

A similar argument shows that any element
\[
(T_{h}\circ R_{f},m)\in \cP^{G}_{Iso}(\Map_{H}(G,X),Y)\times \mM(X)
\]
can be brought into the form
\[
(T_{h}\circ R_{\eta}, R_{f'}m)\in \cP^{G}_{Iso}(\Map_{H}(G,i_{H}^{\ast}S),Y)\times\mM(i_{H}^{\ast}S).
\]

The natural transformation is defined by
\[
(T_{h}\circ N_{\epsilon},m)\mapsto (T_{h}\circ R_{\eta},m);
\]
Hoyer shows that this is well-defined and an isomorphism of Mackey functors \cite[Thm 2.3.3]{HoyerThesis}.
\end{proof}

\begin{corollary}
The composite $N_{H}^{G}\circ i_{H}^{\ast}$ is naturally isomorphic to the left Kan extension along the functor $F(G/H,-)$.
\end{corollary}

This motivates the following definition.

\begin{definition}
If $T$ is a finite $G$ set, then let 
\[
N^{T}\colon \Mackey_{G}\to\Mackey_{G}
\]
be left Kan extension along the functor $S\mapsto F(T,S)$.
\end{definition}

Since 
\[
F(T_{1}\amalg T_{2},S)\cong F(T_{1},S)\times F(T_{2},S),
\]
we have a natural isomorphism of functors
\[
N^{T_{1}\amalg T_{2}}\cong N^{T_{1}}\Box N^{T_{2}}.
\]

\begin{proposition}\label{prop:NTaLeftAdjoint}
For all admissible $G$-sets $T$, we have an adjoint pair of functors on $\cO$-Tambara functors:
\[
N^{T}\leftadjoint m_{T}.
\]
\end{proposition}
\begin{proof}
The assignment $T\mapsto N^{T}$ takes disjoint unions to categorical coproducts, and similarly, $T\mapsto m_{T}$ takes disjoint unions to categorical products. This reduces the proposition to checking on orbits $G/H$. Now we have a natural isomorphism
\[
N_{H}^{G}\circ i_{H}^{\ast}\cong N^{G/H}
\]
arising from the natural isomorphism $F_{H}(G,i_{H}^{\ast}(-))\cong F(G/H,-)$. The result then follows from Theorem~\ref{thm:IdentifyingLeftAdjoint}.
\end{proof}

\bibliographystyle{plain}

\bibliography{Noo}

\begin{thebibliography}{10}

\bibitem{BHmodules}
Andrew~J. Blumberg and Michael~A. Hill.
\newblock {$G$}-symmetric monoidal categories of modules over equivariant
  commutative ring spectra.
\newblock arxiv.org: 1511.07363, 2015.

\bibitem{BHNinfty}
Andrew~J. Blumberg and Michael~A. Hill.
\newblock Operadic multiplications in equivariant spectra, norms, and
  transfers.
\newblock {\em Adv. Math.}, 285:658--708, 2015.

\bibitem{GamKock}
Nicola Gambino and Joachim Kock.
\newblock Polynomial functors and polynomial monads.
\newblock {\em Math. Proc. Cambridge Philos. Soc.}, 154(1):153--192, 2013.

\bibitem{GreenMay}
J.~P.~C. Greenlees and J.~P. May.
\newblock Localization and completion theorems for {$M{\rm U}$}-module spectra.
\newblock {\em Ann. of Math. (2)}, 146(3):509--544, 1997.

\bibitem{HHInversion}
M.~A. Hill and M.~J. Hopkins.
\newblock Equivariant multiplicative closure.
\newblock In {\em Algebraic topology: applications and new directions}, volume
  620 of {\em Contemp. Math.}, pages 183--199. Amer. Math. Soc., Providence,
  RI, 2014.

\bibitem{HHLocalization}
Michael~A. Hill and Michael~J. Hopkins.
\newblock Equivariant symmetric monoidal structures.
\newblock 2012.

\bibitem{HHR}
Michael~A. Hill, Michael~J. Hopkins, and Douglas~C. Ravenel.
\newblock On the non-existence of elements of {K}ervaire invariant one.
\newblock arxiv.org:0908.3724, 2009.

\bibitem{HoyerThesis}
Rolf Hoyer.
\newblock {\em Two topics in stable homotopy theory}.
\newblock PhD thesis, University of Chicago, 6 2014.

\bibitem{KellyLack}
G.~M. Kelly and Stephen Lack.
\newblock Finite-product-preserving functors, {K}an extensions and
  strongly-finitary {$2$}-monads.
\newblock {\em Appl. Categ. Structures}, 1(1):85--94, 1993.

\bibitem{MazurArxiv}
Kristen Mazur.
\newblock An equivariant tensor product on {M}ackey functors.
\newblock arxiv.org: 1508.04062, 2015.

\bibitem{NakaokaIdeals}
Hiroyuki Nakaoka.
\newblock Ideals of {T}ambara functors.
\newblock {\em Adv. Math.}, 230(4-6):2295--2331, 2012.

\bibitem{NakaokaLocalization}
Hiroyuki Nakaoka.
\newblock On the fractions of semi-{M}ackey and {T}ambara functors.
\newblock {\em J. Algebra}, 352:79--103, 2012.

\bibitem{Strickland}
Neil Strickland.
\newblock Tambara functors.
\newblock arxiv.org: 1205.2516, 2012.

\bibitem{Tambara}
D.~Tambara.
\newblock On multiplicative transfer.
\newblock {\em Comm. Algebra}, 21(4):1393--1420, 1993.

\bibitem{Weber}
Mark Weber.
\newblock Polynomials in categories with pullbacks.
\newblock {\em Theory Appl. Categ.}, 30:No. 16, 533--598, 2015.

\end{thebibliography}

\end{document}